\documentclass[12pt]{amsart}
\usepackage{lmodern} 
\usepackage[utf8]{inputenc} 
\usepackage[T1]{fontenc}
\usepackage{amsfonts}
\usepackage{amsmath}
\usepackage{amssymb}
\usepackage{amsthm}
\usepackage[a4paper, left=3cm, right=3cm, top=3cm,
bottom=3cm,dvips]{geometry}
\usepackage{graphicx}
\usepackage[font=small,labelfont=bf]{caption}
\usepackage{color}
\usepackage[normalem]{ulem}

\newtheorem{theorem}{Theorem}[section]
\newtheorem{lemma}[theorem]{Lemma}
\newtheorem{claim}[theorem]{Claim}
\newtheorem{definition}[theorem]{Definition}

\newtheorem{corollary}[theorem]{Corollary}
\newtheorem{observation}[theorem]{Observation}

\title[The polynomial method for 3-path extendability]{The polynomial method for 3-path extendability of list colourings of planar graphs}
\author{Przemys\l{}aw Gordinowicz}
\address{Institute of Mathematics, Lodz University of Technology, \L{}\'od\'z, Poland}
\email{pgordin@p.lodz.pl}

\author{Pawe\l{} Twardowski}
\address{Institute of Mathematics, Lodz University of Technology, \L{}\'od\'z, Poland}
\email{p.twardowski.7@gmail.com}

\begin{document}
	
\begin{abstract}
We restate Thomassen's theorem of 3-extendability~\cite{ThomassenExt}, an extension of the famous planar 5-choosability theorem, in terms of graph polynomials. This yields an Alon--Tarsi equivalent of 3-extendability. 
\end{abstract}
\keywords{planar graph, list colouring, paintability, Alon-Tarsi number.}
\subjclass[2000]{05C10, 05C15, 05C31}

\vspace{150pt}

\maketitle

\section{Introduction}

In his famous paper~\cite{Thomassen} Thomassen proved that every planar graph is 5-choosable. Actually, to proceed with an inductive argument, he proved the following stronger result.
\begin{theorem}[\cite{Thomassen}]\label{thm:thomassen}
	Let $G$ be any plane near-triangulation (every face except the outer one is a triangle) with outer cycle $C$. Let $x$, $y$ be two consecutive vertices on $C$. Then $G$ can be coloured from any list of colours such that the length of lists assigned to $x$, $y$, any other vertex on $C$ and any inner vertex is 1, 2, 3, and 5, respectively. 
\end{theorem}
Taking it differently, $x$ and $y$ can be precoloured in different colours. The situation when a plane graph can be list-coloured from a particular list assignment when some of the vertices on the outer cycle are appropriately precoloured is generally referred to as \emph{extendability}. The assertion of Theorem~\ref{thm:thomassen} can be translated to extendability equivalently as the fact that every planar graph is \emph{2-extendable with respect to vertices $x,y$}, or \emph{2-extendable with respect to the path $\overrightarrow{xy}$}, for any two neighbouring vertices $x,y \in V(C(G))$.

In \cite{ThomassenExt}, Thomassen further explored the concept of path extendability, introducing \emph{3-extendability with respect to the path $\overrightarrow{xyz}$}, or \emph{3-path extendability}. He found out that while it is not true that every plane graph is 3-extendable with respect to any outer cycle path $v_kv_1v_2$, the configuration that prevents such extendability is easily defined. He called such configurations \emph{generalized wheels} (essentially stars inserted into cycles, claws linked by a path, and their combinations, they will be precisely defined and discussed in Section~\ref{sec:wheels}), and proved the following:
\begin{theorem}[\cite{ThomassenExt}]\label{thm:thomassenExt}
	Let $G$ be any plane near-triangulation with outer cycle $C$, and let $v_k$, $v_1$, $v_2$ be three consecutive vertices on $C$. Suppose $v_1,v_2$ and $v_k$ are precoloured, other vertices on $C$ are given lists of length at least 3 and all other vertices are given lists of length at least 5. Then $G$ can be list-coloured unless it contains a subgraph $G'$ which is a generalized wheel whose principal path
	is $v_kv_1v_2$, and all vertices on the outer cycle of $G'$ are on $C$ and have precisely three available
	colours.
\end{theorem}

The Thomassen's 5-choosability theorem was translated into the language of graph polynomials by Zhu in~\cite{Zhu}, showing that Alon-Tarsi number of any planar graph $G$ satisfies $AT(G) \le 5$. His approach utilizes a certain polynomial arising directly from the structure of the graph. This \textit{graph polynomial} is defined as:
$$P(G) = \prod_{uv \in E(G), u<v}(u-v),$$
where the relation $<$ fixes an arbitrary orientation of graph $G$. Here we understand $u$ and $v$ both as the vertices of $G$ and variables of $P(G)$, depending on the context. Notice that the  orientation affects the sign of the polynomial only. Therefore individual monomials and the powers of the variables in each monomials are orientation-invariant. We refer the reader to~\cite{combnull, AlonTarsi, Schautz} for the connection between list colourings and graph polynomials. The use of the polynomials has already proven itself to be a useful and powerful method of dealing with list-colouring problems, like in~\cite{GrytczukZhu} where Grytczuk and Zhu showed that every planar graph $G$ contains a matching $M$ such that $AT(G-M) \le 4$, for which the proof using standard methods is not known. The approach of Zhu may be described in the following form, analogous to Theorem~\ref{thm:thomassen}. 
\begin{theorem}[\cite{Zhu}]\label{thm:zhu}
	Let $G$ be any plane near-triangulation, let $e = xy$ be a boundary edge of $G$. Denote other boundary vertices by $v_1, \dots, v_k$ and inner vertices by $u_1, \dots, u_m$. Then the graph polynomial of $G-e$ contains a non-vanishing monomial of the form $\eta x^0y^0v_1^{\alpha_1}\dots v_k^{\alpha_k} u_1^{\beta_1}\dots u_m^{\beta_m}$ with $\alpha_i \le 2, \beta_j \le 4$ for $i \le k$, $j \le m$.
\end{theorem}
The main tool connecting graph polynomials with list colourings is Combinatorial Nullstellensatz \cite{combnull}. It implies that for every non-vanishing monomial of $P(G)$, if we assign to each vertex of $G$ a list of length greater than the exponent of corresponding variable in that monomial, then such list assignment admits a proper colouring.

The aim of this paper is to provide the graph polynomial analogue of Theorem~\ref{thm:thomassenExt}. The concept of extendability needs a bit of adjustment in terms of the graph polynomial. Obviously there is no way of translating "precolouring" directly into any property of the monomial. By the definition of the polynomial, it is also impossible for the neighbouring vertices to simultaneously have exponents equal to 0 in single monomial --- this is consistent with the list colouring, as the list assignment where two neighboring vertices have lists of length 1 can be invalidated by those lists being identical. In~\cite{ThomassenExt}, Thomassen somehow liberally equals precolouring with lists of length one, with the necessity of differing lists left in the context. Notice, however, the missing edge in the statement of Theorem~\ref{thm:zhu}, which gives the necessary freedom, while still directly implying the full analogue of Theorem~\ref{thm:thomassen}. In the same spirit, in this paper 3-path extendability will be investigated for graphs with missing edges.

Throughout the paper, $\overrightarrow{v_av_b \dots v_k}$ will denote a path in a graph on vertices $v_a$, $v_b, \dots, v_k$, in that order, and $G-\overrightarrow{v_av_b \dots v_k}$ will be understood as deletion of all edges of the path $\overrightarrow{v_av_b \dots v_k}$ from $G$ (while preserving the vertices). To avoid confusion, we will distinguish between two notions and notations of the interior of the cycle. If $C$ is a cycle in graph $G$, then $Int(C)$ will denote the subgraph enclosed by $C$ including the cycle itself, while $int(C)$ will denote only the subgraph inside the cycle, excluding the vertices of $C$ and incident edges. Similarly, $int(G)$ will denote a subgraph of $G$ induced by the vertices not on the outer cycle. All considered graphs are simple, undirected, and finite. For background in graph theory, see~\cite{West}. 

\section{Technical remarks}

We will begin with a minor technical detail needed for the proof of the main result. The starting point is the following lemma which appeared in~\cite{GT2021}

\begin{lemma}\label{lem:l2}\cite[Lemma 3.5]{GT2021}
	Let $G,G'$ be any two graphs, such that $V(G) = \lbrace x, v_1, \dots, v_n \rbrace $, $V(G') = \lbrace x, u_1, \dots, u_m \rbrace$, $V(G) \cap V(G')=\lbrace x \rbrace$. Let $G'' = G \cup G'$. Suppose there are non-vanishing monomials $\eta x^\alpha \Pi v_i^{\alpha_i}$ and $\eta' x^\beta \Pi u_j^{\beta_j}$ in $P(G)$ and $P(G')$ respectively. Then in $P(G'')$ there is a non-vanishing monomial
	$$A(G'') = \eta \eta' x^{\alpha + \beta}  \prod v_i^{\alpha_i} \prod u_j^{\beta_j}.$$
\end{lemma}

As a consequence of Lemma~\ref{lem:l2}, we obtain the following corollaries.

\begin{corollary}\label{cor:00}
	Let $G,G'$ be any two graphs, such that $V(G) = \lbrace x, y, v_1, \dots, v_n \rbrace $, $V(G') = \lbrace x, y, u_1, \dots, u_m \rbrace$, $V(G) \cap V(G')=\lbrace x,y \rbrace$ and $E(G) \cap E(G')=\lbrace xy \rbrace$. Let $G'' = G \cup G'$. Suppose there are non-vanishing monomials $\eta x^0 y^\alpha \Pi v_i^{\alpha_i}$ and $\eta' x^0 y^\beta \Pi u_j^{\beta_j}$ in $P(G)$ and $P(G')$ respectively. Then in $P(G'')$ there is a non-vanishing monomial
	$$A(G'') = \eta \eta' x^0 y^{\alpha + \beta - 1}  \prod v_i^{\alpha_i} \prod u_j^{\beta_j}.$$
\end{corollary}
\begin{proof}
	Consider $G_0''=G''-x$. It has a cutvertex $y$ and splits on it into $G_0=G-x$ and $G_0'=G'-x$. As $P(G)$ contained a non-vanishing monomial $\eta x^0 y^\alpha \Pi v_i^{\alpha_i}$, $P(G_0)$ contains non-vanishing
	$$\eta y^{\alpha-1} \prod_{v_i \nsim x} v_i^{\alpha_i} \prod_{v_i \sim x} v_i^{\alpha_i-1},$$
	and similarly $P(G_0')$ contains a non-vanishing monomial
	$$\eta' y^{\beta-1} \prod_{u_j \nsim x} u_j^{\beta_j} \prod_{u_j \sim x} u_j^{\beta_j-1}.$$
	Therefore, by Lemma~\ref{lem:l2}, $P(G_0'')$ contains a non-vanishing monomial
	$$\eta \eta' y^{\alpha + \beta - 2} \prod_{v_i \nsim x} v_i^{\alpha_i} \prod_{v_i \sim x} v_i^{\alpha_i-1} \prod_{u_j \nsim x} u_j^{\beta_j} \prod_{u_j \sim x} u_j^{\beta_j-1},$$
	and as
	$$N_{G''}(x)=N_{G}(x) \cup N_{G'}(x),$$
	in $P(G'')$ contains a non-vanishing monomial
	$$\eta \eta' x^0 y^{\alpha + \beta - 1} \prod v_i^{\alpha_i} \prod u_j^{\beta_j}.$$
\end{proof}
\begin{corollary}\label{cor:00_2}
	Let $G,G'$ be any two graphs, such that $V(G) = \lbrace x, y, v_1, \dots, v_n \rbrace $, $V(G') = \lbrace x, y, u_1, \dots, u_m \rbrace$, $V(G) \cap V(G')=\lbrace x,y \rbrace$ and $xy \in E(G)$, $xy \notin E(G')$. Let $G'' = G \cup G'$. Suppose there are non-vanishing monomials $\eta x^0 y^\alpha \Pi v_i^{\alpha_i}$ and $\eta' x^0 y^\beta \Pi u_j^{\beta_j}$ in $P(G)$ and $P(G')$ respectively. Then in $P(G'')$ there is a non-vanishing monomial
	$$A(G'') = \eta \eta' x^0 y^{\alpha + \beta}  \prod v_i^{\alpha_i} \prod u_j^{\beta_j}.$$
\end{corollary}

\section{Short outer cycles}
	
Following the footsteps of \cite{ThomassenExt}, we will prove a simple, but interesting, fact about the list colourings of planar graphs --- that if the length of the outer cycle of the graph is short (no more than 5), then in most situations the problems of list assignment of the outer cycle and the interior of the graph can essentially be separated. It also serves as a polynomial analogue to the theorem of Bohme, Mohar and Stiebitz~\cite{BMS}.
	
\begin{theorem}\label{thm:smallcycle}
	Let $G$ be a chordless near-triangulation with outer cycle $C(G) = \lbrace v_1,\dots,v_i \rbrace, i \in \lbrace 3,4,5 \rbrace$, with $I(G) = V(G)\setminus V(C(G)) = \lbrace u_1,\dots, u_j\rbrace$. Suppose that in case $i = 5$ there is no vertex in $I(G)$ adjacent to all vertices of $C(G)$. Then for any monomial $\eta v_1^{\alpha_1}\dots v_i^{\alpha_i}$, $i \in \lbrace 3,4,5 \rbrace, \eta \ne 0$, in $P(C(G))$ there is a monomial
	$$\eta' v_1^{\alpha_1}\dots v_i^{\alpha_i} u_1^{\beta_1}\dots u_j^{\beta_j}$$
	in $P(G)$ with $\eta' \ne 0$ and $\beta_k \le 4$ for all $1 \le k \le j$.
\end{theorem}
In order to prove this theorem, it actually suffices to prove the following.
\begin{theorem}\label{thm:smallcycle_0}
	Let $G$ be a chordless near-triangulation with outer cycle $C(G) = \lbrace v_1,\dots,v_i \rbrace, i \in \lbrace 3,4,5 \rbrace$, with $I(G) = V(G)\setminus V(C(G)) = \lbrace u_1,\dots, u_j\rbrace$. Suppose that in case $i = 5$ there is no vertex in $I(G)$ adjacent to all vertices of $C(G)$. Then there is a monomial
	$$\eta v_1^0\dots v_i^0 u_1^{\beta_1}\dots u_j^{\beta_j}$$
	in $P(G-E(C(G)))$ with $\eta \ne 0$ and $\beta_k \le 4$ for all $1 \le k \le j$.
\end{theorem}
\begin{proof}
	Suppose $|C(G)| = 3$. Consider the graph
	$$G_1=(G - v_3) - v_1v_2.$$
	By Theorem~\ref{thm:zhu}, in $P(G_1)$ there is a non-vanishing monomial
	$$A(G_1)=\eta_1 v_1^0v_2^0 u_1^{\beta_1}\dots u_j^{\beta_j}.$$
	Moreover, as due to planarity every vertex $u_k$ neighboring $v_3$ in $G$ is on an outer cycle of $G_1$, we have $\beta_k\le 2$ for each of those vertices, and $\beta_k\le 4$ for any other vertex $u_k\in I(G)$ in $A(G_1)$. Construct $G_2$ from $G_1$ by adding $v_3$ and linking it to every vertex from $I(G)$ that neighbors $v_3$ in $G$. As for every neighbor of $v_3$ we had an exponent no larger than $2$ in $A(G_1)$, there is a nonvanishing monomial $A(G_2) \in P(G_2)$ such that
	$$A(G_2)=\eta v_1^0v_2^0v_3^0 u_1^{\beta_1}\dots u_j^{\beta_j}$$
	$\beta_k\le 3$ for every $u_k$ neighboring $v_3$, and $\beta_k\le 4$ for any other $u_k\in I(G)$. It only remains to notice that
	$$G_2=G-E(C(G)).$$
	
	Suppose now that $|C(G)| = 4$. Let
	$$G_1=G - \lbrace v_3, v_4 \rbrace - v_1v_2.$$
	In $P(G_1)$ there is non-vanishing monomial $A(G_1)$ such that
	$$A(G_1)=\eta_1 v_1^0v_2^0 u_1^{\beta_1}\dots u_j^{\beta_j},$$
	and for every $u_k$ neighboring either $v_3$ or $v_4$ in $G$ we have $\beta_k\le 2$ and for any other $u_k$ we have $\beta_k\le 4$. Add back $v_3$ to its neighbors from $G$. Call this newly created graph $G_2$. In $P(G_2) $ there is non-vanishing $A(G_2)$ such that
	$$A(G_2)=\eta v_1^0v_2^0v_3^0 u_1^{\beta_1}\dots u_j^{\beta_j},$$
	for every $u_k$ neighboring $v_3$ we have $\beta_k\le 3$, for every $u_k$ neighboring $v_4$ and not $v_3$ we have $\beta_k\le 2$, and for any other $u_k$ we have $\beta_k\le 4$. After constructing $G_3$ by adjoining $v_4$ in the same manner, we have in $P(G_3)$ non-vanishing $A(G_3)$ such that
	$$A(G_3)=\eta v_1^0v_2^0v_3^0v_4^0 u_1^{\beta_1}\dots u_j^{\beta_j}$$
	$\beta_k\le 3$, for any $u_k$ that neighbors exactly one of $v_3, v_4$, $\beta_k\le 4$ for any other $u_k$, and
	$$G_3=G-E(C(G)).$$
	
	The remaining case is when $|C(G)| = 5$. Suppose at first that there is no vertex in $I(G)$ that neighbors more than three of the vertices from $C(G)$. Observe that due to planarity, it is possible to find a triple of consecutive vertices that do not share a common neighbor. Relabel vertices of $C(G)$ in counter-clockwise order in a manner such that $v_3,v_4,v_5$ form this kind of triple. Starting with
	$$G_1=G- \lbrace v_3, v_4, v_5 \rbrace - v_1v_2$$
	and proceeding as in previous point we end up with $G_3$ such that there is non-vanishing $A(G_3)$ in $P(G_3)$ with
	$$A(G_2)=\eta v_1^0v_2^0v_3^0 u_1^{\beta_1}\dots u_j^{\beta_j},$$
	$\beta_k\le 2$ for every $u_k$ that in $G$ neighbors $v_5$ and neither of $v_3,v_4$, $\beta_k\le 3$ for every $u_k$ neighboring exactly one of $v_3, v_4$, and $\beta_k\le 4$ for any other $u_k$. Add back $v_5$, obtaining
	$$G_4=G-E(C(G)).$$
	Notice that due to the fact that no neighbor of $v_5$ was adjacent to both $v_3$ and $v_4$, for every neighbor of $v_5$ we have $\beta_k\le 3$ in $A(G_3)$. Therefore in $P(G_4)$ we have non-vanishing monomial:
	$$A(G_4)=\eta v_1^0v_2^0v_3^0v_4^0v_5^0 u_1^{\beta_1}\dots u_j^{\beta_j}, $$
	with $\beta_k \le 4$.
	
	If there is a vertex in $I(G)$ that neighbors more than three of the vertices from $C(G)$, then by the assumption it neighbors excatly four of them. Notice moreover, that if there were two such vertices, then it contradicts planarity, as one of them would restrict the other to the face with at most three vertices on the outer cycle. Therefore it is a unique vertex $u_l$. Again relabel vertices of $C(G)$ in counter-clockwise order such that $v_5$ is the sole non-neighbor of $u_l$ on the outer cycle. Split $G$ along the path $\overrightarrow{v_4u_lv_1}$ into $G_1$ and $G_2$, where $v_5 \in V(G_2)$. Consider $G_1-\overrightarrow{v_1v_2v_3v_4}$. It consists of a star $S_4$ centered at $u_l$ and three subgraphs enclosed by triangles $v_1v_2u_l$, $v_2v_3u_l$ and $v_3v_4u_l$ in $G$. As such, as we already established the theorem for $i=3$, in $P(G_1-\overrightarrow{v_1v_2v_3v_4})$ there is a monomial
	$$A_1=\eta_1v_1^0v_2^0v_3^0v_4^0u_l^4 \prod u_k^{\beta_k}, \beta_k \le 4.$$
	Now, $G_2$ is a chordless near-triangulation with the outer cycle of length 4, therefore $P(G_2-E(C(G_2)))$ contains a monomial
	$$A_2=\eta_2v_1^0v_4^0v_5^0u_l^0 \prod u_k^{\beta_k}, \beta_k \le 4.$$
	A brief examination yields 
	$$G-E(C(G))=(G_1-\overrightarrow{v_1v_2v_3v_4}) \cup (G_2-E(C(G_2))),$$
	and therefore:
	$$P(G-E(C(G)))=P((G_1-\overrightarrow{v_1v_2v_3v_4}))P((G_2-E(C(G_2)))).$$
	Hence
	$$A_1A_2=\eta_1\eta_2v_1^0v_2^0v_3^0v_4^0v_5^0u_l^4 \prod u_k^{\beta_k},\beta_k \le 4,$$
	is contained in $P(G-E(C(G)))$, and as $v_1,v_4$ and $u_l$ are the only common vertices of $G_1$ and $G_2$, and exponents of each of these vertices is equal to 0 in $A_2$, $A_1A_2$ cannot vanish in $P(G-E(C(G)))$.
\end{proof}
It can be immediately seen that this applies not only to the planar graphs with short outer cycle, but to all planar graphs whose outer cycle induced subgraph has "small faces". It can be put rigorously in the form of the following corollary.
\begin{corollary}
	Let $G$ be a plane graph with outer cycle $C$ such that every face of the subgraph $C(G)$ induced by $C$ other than the outer face is bounded by no more than $5$ edges. Suppose also that no vertex not on $C$ neighbors more than $4$ vertices on $C$. Then for every monomial $A$ in $P(C(G))$ there is a non-vanishing monomial $\eta A \Pi u_i^{\beta_i}$, with $u_i \in V(G) \setminus V(C(G))$ and $\beta_i \le 4$, in $P(G)$.
\end{corollary}

\section{Outerplanar near-triangulations with universal vertex}\label{sec:sec3}

In \cite{GT2021}, the authors used the graph polynomials to provide a characterization of the outerplanar graphs with respect to their \emph{($i,j$)-$xy$-extendability}, providing a polynomial analogue to the work of Hutchinson~\cite{Hutchinson} (which was later extended to planar graphs by Postle and Thomas~\cite{PoThI,PoThII,PoThIII}). This type of extendability is defined as follows: let $x,y$ be any two vertices of $G$, equipped with the lists of length $i$ and $j$, respectively. Then the list assignment can be extended in a specific manner (in case of outerplanar graphs all other vertices are given lists of length 3) such that $G$ is list-colorable with this assignment. The characterization is summed up in the form of the following theorem.

\begin{theorem}\cite[Theorem 1.4]{GT2021} \label{thm:outer}
	Let $G$ be any outerplanar graph with $V(G) = 
	\{x, y, v_1, \dots, v_n\}$. Then in $P(G)$ there is a non-vanishing monomial of the form $\eta x^\beta y^\gamma \prod_{i=1}^n v_i^{\alpha_i}$ with $\alpha_i \le 2$, $\beta, \gamma \le 1$ satisfying:
	\begin{enumerate}
		\item $\beta = \gamma = 1$ when every proper $3$-colouring $\mathcal{C}$ of $G$ forces $\mathcal{C}(x) = \mathcal{C}(y)$;
		\item $\beta + \gamma = 1$ when every proper $3$-colouring $\mathcal{C}$ of $G$ forces  $\mathcal{C}(x) \neq \mathcal{C}(y)$; 
		\item $\beta = \gamma = 0$ otherwise.
	\end{enumerate}
\end{theorem}

\begin{figure}[htb]
	\begin{center}
		\includegraphics[width = 0.9\textwidth]{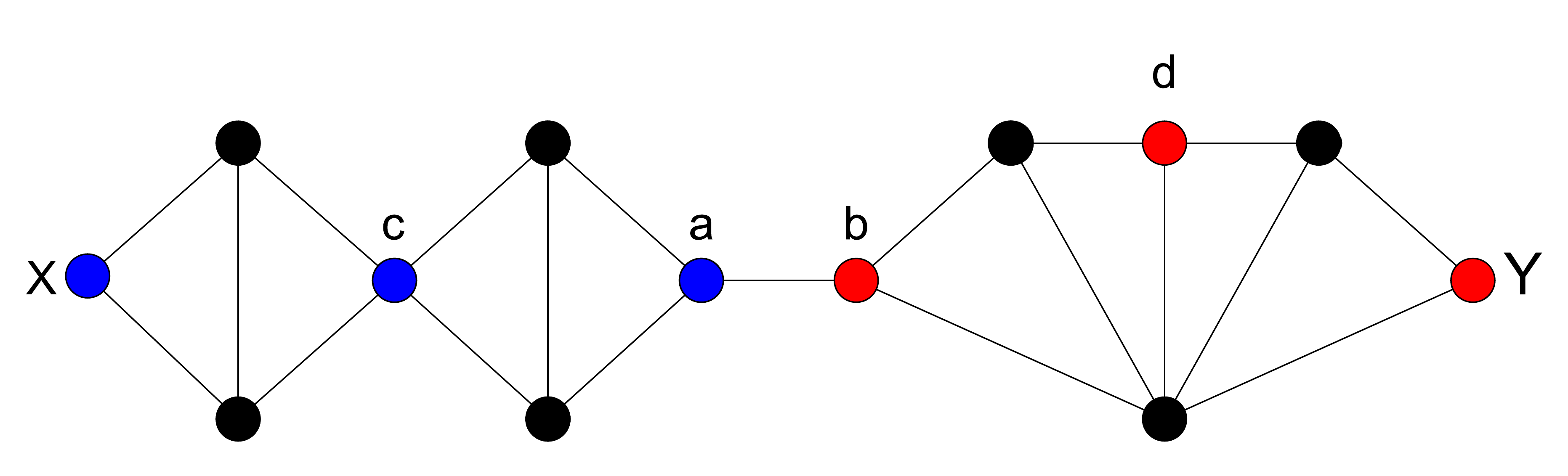}
	\end{center}
	\caption{An example of a graph satisfying conditions of point $ii)$ of Theorem~\ref{thm:outer}. When 3-colouring the graph, vertices $a$ and $b$ need to be in different colours. Vertices $x$ and $c$ are in the same colour class as $a$ (an example of the chain of diamonds), while $y$ and $d$ are in the same colour class as $b$ (the diamonds are linked along an edge). Therefore $x$ and $y$ have different colours in every proper 3-colouring of the graph. The black vertices are yet to be coloured.} \label{fig:colouring}
\end{figure}

A particular type of outerplanar graphs turns out to be deeply linked to the problem of the 3-path extendability, especially to the concept of generalized wheels. These are the so called \emph{outerplanar near-triangulations with universal vertex}. It is an (indexed) family of 2-connected outerplanar near-triangulations having a \emph{universal vertex} (i.e. a vertex neighboring all other vertices).
Its fairly straighforward to deduce from the properties of the 2-connected outerplanar triangulations that apart from the universal vertex such graphs have precisely two vertices of degree 2, and the degrees of all remaining vertices are equal to 3.

\begin{figure}[htb]
	\begin{center}
		\includegraphics[scale=0.55]{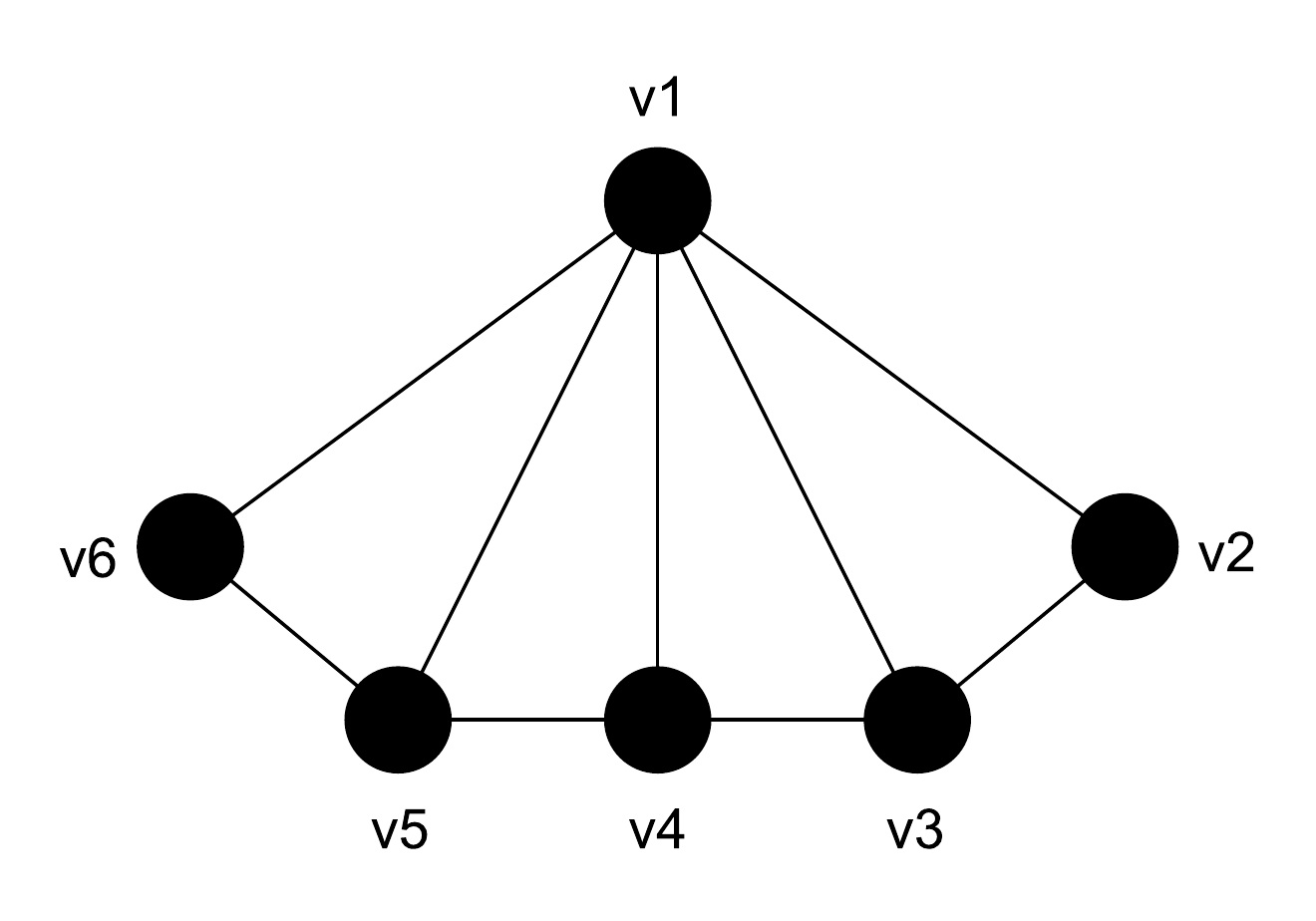}	
	\end{center}
	\caption{An example of outerplanar near-triangulation with universal vertex $v_1$. Note that in this example $\mathcal{C}(v_2)=\mathcal{C}(v_6)$ in every proper 3-colouring $\mathcal{C}$ of the graph.} \label{fig:univ_vertex}
\end{figure}

The 2-connected outerplanar near-triangulations are more restricted in terms of ($i,j$)-$xy$-extendability than general outerplanar graphs. For example, they admit a unique (up to the choice of colours) 3-colouring, therefore always fall under either point (i) or (ii) of Theorem~\ref{thm:outer}. Hence they cannot be ($1,1$)-$xy$-extendable. For the needs of this paper, however, a much thorough examination of the monomials available (and unavailable) in the polyomials of such near-triangulations with universal vertex was needed. The ensuing flurry of technical lemmas ultimetely serve the purpose of the deep understanding of the polynomials of generalized wheels, which is indispensable in the study of polynomial 3-path extendability of planar graphs. We begin with the specific symmetry of the polynomials of a subfamily of outerplanar near-triangulations with universal vertex. To further single out the universal vertex, we will denote it by $u$ in this section (with the remaining vertices labelled $v_1,\dots,v_k$), in contrast to subsequent sections where it will be denoted as $v_1$.

\begin{figure}[htb]
	\begin{center}
		\includegraphics[scale=0.55]{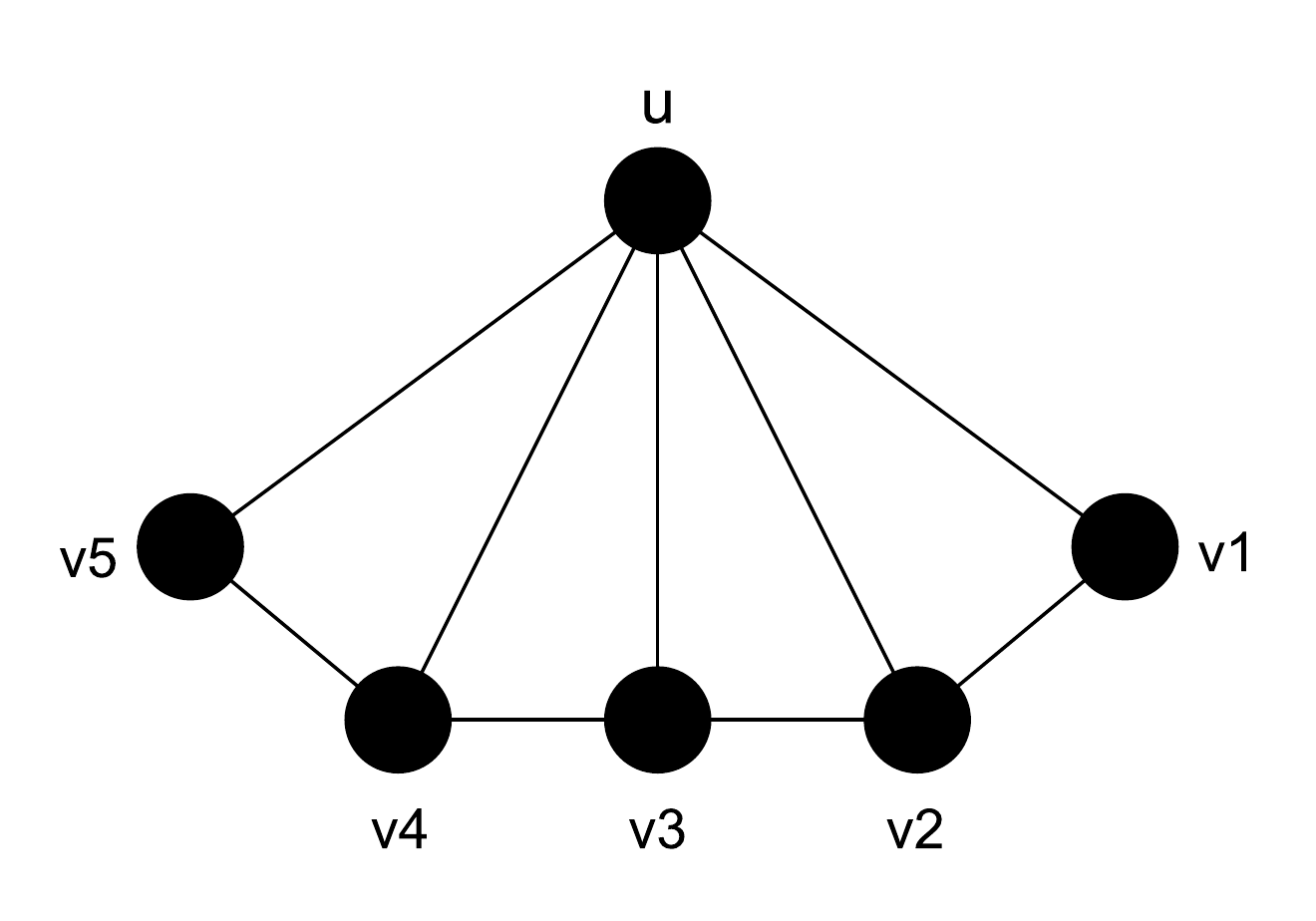}
	\end{center}
	\caption{Labelling of the graph from Figure~\ref{fig:univ_vertex} as used in Lemmas~\ref{lem:signsymmetry}~---~\ref{lem:odd_wheels} and Corollary~\ref{cor:all_121_odd}.} \label{fig:univ_vertex_2}
\end{figure}

\begin{lemma}\label{lem:signsymmetry}
	Let $G$ be a 2-connected outerplanar near-triangulation such that
	$$V(G)=\lbrace u,v_1,\dots,v_{2k} \rbrace,$$
	$$N_G(u)=\lbrace v_1,\dots,v_{2k} \rbrace$$
	and
	$$deg_G(v_1)=2=deg_G(v_{2k}).$$
	If $P(G)$ contains a non-vanishing monomial $\eta u^{\alpha_u}\Pi a_i^{\alpha_i}$, then it also contains $-\eta u^{\alpha_u}\Pi a_i^{\alpha_{2k-i+1}}$.
\end{lemma}
\begin{proof}
	If $k=1$, then $G$ is the triangle $uv_1v_2$, hence:
	$$P(G)=u^2v_1^1v_2^0-u^2v_1^0v_2^1-u^1v_1^2v_2^0+u^1v_1^0v_2^2+u^0v_1^2v_2^1-u^0v_1^1v_2^2.$$
	Suppose $k>1$. Observe that $Aut(G)$ has two elements, $id$ and $\sigma$, where $\sigma$ is given by the mapping:
	$$u \mapsto u,$$
	$$v_i \mapsto v_{2k-i+1},$$
	with the subsequent mapping of edge orientations:
	$$uv_i \mapsto uv_{2k-i+1},$$
	$$v_iv_{i+1} \mapsto v_{2k-i+1}v_{2k-i}.$$
	As
	$$P(G)=\prod_{i=1}^{2k}(u-v_i)\prod_{i=1}^{2k-1}(v_i-v_{i+1}),$$
	then
	$$P(\sigma(G))=\prod_{i=1}^{2k}(u-v_{2k-i+1})\prod_{i=1}^{2k-1}(v_{2k-i+1}-v_{2k-i})=$$
	$$=\prod_{i=1}^{2k}(u-v_i)\prod_{i=1}^{2k-1}-(v_{2k-i}-v_{2k-i+1})\stackrel{\text{\tiny j=2k-i}}{=}(-1)^{2k-1}\prod_{i=1}^{2k}(u-v_i)\prod_{j=1}^{2k-1}(v_{j}-v_{j+1})=-P(G).$$
	Hence for every monomial $\eta u^{\alpha_u}\Pi a_i^{\alpha_i}$ in $P(G)$ there is a monomial $-\eta u^{\alpha_u}\Pi a_i^{\alpha_i}$ in $P(\sigma(G))$. By relabelling the vertices of $\sigma(G)$:
	$$u \rightarrow u,$$
	$$v_i \rightarrow v_{k-i+1}$$
	we obtain $G$, and therefore as there was a monomial $-\eta u^{\alpha_u}\Pi a_i^{\alpha_i}$ in $P(\sigma(G))$, we have $-\eta u^{\alpha_u}\Pi a_i^{\alpha_{k-i+1}}$ in $P(G)$.
\end{proof}
The following lemma may be directly deduced from either Theorem~\ref{thm:zhu} or Theorem~\ref{thm:outer}, but we give a simple, independent proof here.
\begin{lemma}\label{lem:012}
	Let $G$ be a 2-connected outerplanar near-triangulation such that
	$$V(G)=\lbrace u,v_1,\dots,v_k \rbrace, k>1,$$
	$$N_G(u)=\lbrace v_1,\dots,v_k \rbrace$$
	and
	$$deg_G(v_1)=2=deg_G(v_k).$$
	Then $P(G)$ contains non-vanishing monomials $\eta_1 u^1 v_1^0 \Pi v_i^2$ and $\eta_2 u^1 v_k^0 \Pi v_i^2$.
\end{lemma}
\begin{proof}
	We begin with a path $\overrightarrow{v_2\dots v_k}$. Its polynomial obviously contains a non-vanishing monomial $\eta_1 v_2^0\Pi v_i^1$. Now adjoin $u$ to all vertices $v_i$. In the polynomial of the resulting graph we can find a monomial $\eta_1 u^0v_2^1\Pi v_i^2$. To conclude the construction of $G$ we attach $v_1$ to $v_2$ and $u$, and clearly in $P(G)$ there is a monomial $\eta_1 u^1 v_1^0 \Pi v_i^2$. The case $\eta_2 u^1 v_k^0 \Pi v_i^2$ is symmetric.
\end{proof}
As \cite{GT2021} only concerned extendability of the outerplanar graphs, all monomials where any of the exponents exceeded 2 (hence any of the vertices was given a list of length greater than 3) were left out as unacceptable. However, as interior vertices of planar graphs may be assigned a list of length up to 5 in extendability context, and the outerplanar near-triangiulations are ultimately only building blocks of graphs we are interested in, specific instances of such monomials need to be investigated here.
\begin{lemma}\label{lem:030}
	Let $G$ be a 2-connected outerplanar near-triangulation such that $$V(G)=\lbrace u,v_1,\dots,v_k \rbrace, k>1,$$
	$$N_G(u)=\lbrace v_1,\dots,v_k \rbrace$$
	and
	$$deg_G(v_1)=2=deg_G(v_k).$$
	Then $P(G)$ contains non-vanishing monomials $\eta_l u^2 v_1^0v_k^0v_l^3 \Pi v_i^2$ for every $2\le l < k$.
\end{lemma}
\begin{proof}
	Start with a path $\overrightarrow{v_2\dots v_{k-1}}$ and fix $l$ appropriately. As the polynomial of $\overrightarrow{v_2\dots v_l}$ contians a monomial $\eta_1 v_2^0\Pi v_i^1$, and the polynomial of $\overrightarrow{v_l\dots v_{k-1}}$ contians a monomial $\eta_2 v_{k-1}^0\Pi v_i^1$, the polynomial of $\overrightarrow{v_2\dots v_{k-1}}$ contains the monomial $\eta v_2^0v_{k-1}^0v_l^2 \Pi v_i^1$, which is non-vanishing by Lemma~\ref{lem:l2}. We adjoin $u$ to all vertices $v_i$ to find a monomial $\eta_l u^0v_2^1v_{k-1}^1v_l^3 \Pi v_i^2$ in the polynomial of the resulting graph. It now suffices to add back $v_1$ and $v_k$ and multiply the polynomial by $u^2v_1^0v_2^1v_{k-1}^1v_k^0$ to generate the monomial
	$$\eta_l u^2v_1^0v_2^2v_{k-1}^2v_k^0v_l^3 \prod v_i^2,$$
	and it does not vanish as $u^2v_1^0v_2^1v_{k-1}^1v_k^0$ is the only monomial in the polynomial of the added graph with $v_1^0v_k^0$.
\end{proof}
\begin{lemma}\label{lem:final}
	Let $G$ be a 2-connected outerplanar near-triangulation such that $$V(G)=\lbrace u,v_1,\dots,v_k \rbrace, k>3,$$
	$$N_G(u)=\lbrace v_1,\dots,v_k \rbrace$$
	and
	$$deg_G(v_1)=2=deg_G(v_k).$$
	Then $P(G)$ contains a non-vanishing monomial $\eta u^3 v_1^1v_2^1v_k^0\Pi v_i^2$.
\end{lemma}
\begin{proof}
	Suppose at first that $k$ is even. Then $\mathcal{C}(v_1) = \mathcal{C}(v_{k-1})$ for every 3-colouring $\mathcal{C}$ of $G$, hence by Theorem~\ref{thm:outer} and \cite[Theorem 2.1]{GT2021} $P(G-v_k)$ contains a non-vanishing monomial  $\eta_l u^2 v_1^1v_2^1v_{k-1}^1\Pi v_i^2$, which in turn gives a non-vanishing $\eta u^3 v_1^1v_2^1v_k^0\Pi v_i^2$ in $P(G)$.
	
	Now assume $k$ is odd. Then $\mathcal{C}(v_1) = \mathcal{C}(v_{k-1})$ for every 3-colouring $\mathcal{C}$ of $G$, hence $P(G-v_k)$ contains a non-vanishing monomial  $\eta' u^2 v_1^1v_2^1v_k^1\Pi v_i^2$. Therefore, $P(G-v_k)$ contains either $\eta u^2 v_1^1v_2^1v_{k-1}^1\Pi v_i^2$ or $\eta u^1 v_1^1v_2^1v_{k-1}^2\Pi v_i^2$. Suppose the latter is true. It would imply that in $P(G-\lbrace v_{k-1},v_k \rbrace)$ there is a monomial $\eta u^1 v_1^1v_2^1v_{k-2}^2\Pi v_i^2$, and continuing in this manner, we would come to the conclusion that there is a non-vanishing monomial $\eta u^1v_1^1v_2^1$. But that leads to a contradiction as such monomials are known to vanish in the polynomials of triangles. Hence in $P(G-v_k)$ there is a non-vanishing monomial $\eta u^2 v_1^1v_2^1v_{k-1}^1\Pi v_i^2$, and consequently there is $\eta u^3 v_1^1v_2^1v_k^0\Pi v_i^2$ in $P(G)$.
\end{proof}
The existence of several types of monomials depends on the parity of the order of the considered graph. This stems form the fact that this parity determines the outcome of the colouring condition as given in the statement of Theorem~\ref{thm:outer}. We take on the odd order graphs first.
\begin{lemma}\label{lem:no_030_even}
	Let $G$ be a 2-connected outerplanar near-triangulation such that
	$$V(G)=\lbrace u,v_1,\dots,v_{2k} \rbrace, k>1,$$
	$$N_G(u)=\lbrace v_1,\dots,v_{2k} \rbrace,$$
	$$deg_G(v_1)=2=deg_G(v_{2k})$$
	and vertices $v_i$ are ordered counter-clockwise. Then $u^3v_1^0v_{2k-1}^0\Pi v_i^2$ vanishes in $P(G)$.
\end{lemma}
\begin{proof}
	Observe at first, that the condition that there is an even number of vertices $v_i$ is equivalent to the condition that for every 3-colouring $\mathcal{C}$ of $G$ we have $\mathcal{C}(v_1) \neq \mathcal{C}(v_{2k})$. Therefore if we consider $G'=G-v_1$, in every 3-colouring $\mathcal{C}'$ of $G'$, $\mathcal{C}'(v_2) = \mathcal{C}'(v_{2k})$. If there was a non-vanishing monomial $\eta u^3v_1^0v_{2k-1}^0\Pi v_i^2$ in $P(G)$, then $P(G')$ would contain a non vanishing  monomial $\eta u^2v_2^1v_{2k-1}^0\Pi v_i^2$, but that would contradict Theorem~\ref{thm:outer}.
\end{proof}
\begin{lemma}\label{lem:all_040_even}
	Let $G$ be a 2-connected outerplanar near-triangulation such that $$V(G)=\lbrace u,v_1,\dots,v_{2k} \rbrace,$$
	$$N_G(u)=\lbrace v_1,\dots,v_{2k} \rbrace,$$
	$$deg_G(v_1)=2=deg_G(v_{2k})$$
	and vertices $v_i$ are ordered counter-clockwise. Let $1\le l<k$. Then $P(G)$ contains  monomials $\eta u^4v_1^0v_{2k}^0v_{2l}^1 \Pi v_i^2$ and $-\eta u^4v_1^0v_{2k}^0v_{2l+1}^1 \Pi v_i^2$, where $|\eta| = 1$.
\end{lemma}
\begin{proof}
	Consider $G-v_{2l}v_{2l+1}$. This graph has $u$ as its cutvertex that we use to split it into components $G_1$ and $G_2$, where $v_1 \in V(G_1)$. For every 3-colouring $\mathcal{C}_1$ of $G_1$ we have $\mathcal{C}_1(v_1) \neq \mathcal{C}_1(v_{2l})$ and for every 3-colouring $\mathcal{C}_2$ of $G_2$ we have  $\mathcal{C}_2(v_{2l+1}) \neq \mathcal{C}_2(v_{2k})$. Therefore by Theorem~\ref{thm:outer} $P(G_1)$ contains a non-vanishing monomial $\eta_1 u^2v_1^0v_{2l}^1 \Pi v_i^2$ and $P(G_2)$ contains a non-vanishing monomial $\eta_2 u^2v_{2l+1}^1v_{2k}^0 \Pi v_i^2$. By Lemma~\ref{lem:l2}, these combine into $\eta u^4v_1^0v_{2k}^0v_{2l}^1v_{2l+1}^1\Pi v_i^2$ in $P(G-v_{2l}v_{2l+1})$, with
	$$\eta=\eta_1\eta_2 \neq 0.$$
	Clearly,
	$$P(G)=P(G-v_{2l}v_{2l+1})(v_{2l}-v_{2l+1})=P(G-v_{2l}v_{2l+1})v_{2l}-P(G-v_{2l}v_{2l+1})v_{2l+1},$$
	hence $P(G)$ contains monomials $\eta u^4v_1^0v_{2k}^0v_{2l}^2v_{2l+1}^1\Pi v_i^2$ (in $P(G-v_{2l}v_{2l+1})v_{2l}$) and $-\eta u^4v_1^0v_{2k}^0v_{2l}^1v_{2l+1}^2\Pi v_i^2$ (in $-P(G-v_{2l}v_{2l+1})v_{2l+1}$). Suppose that the former vanishes in $P(G)$, which would imply the existence of monomial $-\eta u^4v_1^0v_{2k}^0v_{2l}^2v_{2l+1}^1\Pi v_i^2$ in $-P(G-v_{2l}v_{2l+1})v_{2l+1}$, and consequently $\eta u^4v_1^0v_{2k}^0v_{2l}^2v_{2l+1}^0\Pi v_i^2$ in $P(G-v_{2l}v_{2l+1})$. But as $P(G_2)$ is homogeneous, it would have to contain a monomial $\eta_2' u^3v_{2l+1}^0v_{2k}^0 \Pi v_i^2$, which is not the case. By the symmetric argument we prove the non-vanishing of the second monomial. The fact that $|\eta| = 1$ comes from the fact that the absolute values coefficients of non-vanishing admissible monomials was shown to be 1 by~\cite[Theorem 2.1]{GT2021}.
\end{proof}
As a direct consequence of the two preceding lemmas, we get the following combined lemma.
\begin{lemma}\label{lem:even_wheels}
	Let $G$ be an outerplanar near-triangulation on $$V(G)=\lbrace u,v_1,\dots,v_{2k} \rbrace, k>1,$$
	$$N_G(u)=\lbrace v_1,\dots,v_{2k} \rbrace,$$
	$$deg_G(v_1)=2=deg_G(v_{2k})$$
	and vertices $v_i$ ordered counter-clockwise. Then $P(G)$ contains a non-vanishing monomial $\eta_l u^4v_1^0v_{2k-1}^0v_l^1\Pi v_i^2$, $|\eta_l| = 1$, for every $2 \le l < 2k$, while $u^3v_1^0v_{2k-1}^0\Pi v_i^2$ vanishes.
\end{lemma}
We now turn our attention to the near-triangulations of an even order.
\begin{lemma}\label{lem:odd_wheels}
	Let $G$ be a 2-connected outerplanar near-triangulation such that
	$$V(G)=\lbrace u,v_1,\dots,v_{2k-1} \rbrace, k>1,$$
	$$N_G(u)=\lbrace v_1,\dots,v_{2k-1} \rbrace,$$
	$$deg_G(v_1)=2=deg_G(v_{2k-1})$$
	and vertices $v_i$ are ordered counter-clockwise. Then $P(G)$ contains a non-vanishing monomial $\eta u^3v_1^0v_{2k-1}^0\Pi v_i^2$. Moreover, $\eta_l u^4v_1^0v_{2k-1}^0v_l^1\Pi v_i^2$ does not vanish for every odd $3 \le l < 2k-1$ and vanishes for every even $2 \le l < 2k$, and $|\eta|=|\eta_l|=1$.
\end{lemma}
\begin{proof}
	Observe at first, that the condition that there is an odd number of vertices $v_i$ is equivalent to the condition that for every 3-colouring $\mathcal{C}$ of $G$ we have $\mathcal{C}(v_1)=\mathcal{C}(v_{2k-1})$. For $k=2$, a straightforward calculation shows that $P(G)$ contains a non-vanishing monomial $-u^3v_2^2$ (for the vertex ordering $u<v_1<v_2<v_3$), while $\eta u^4v_2^1$ is not feasible as $deg_G(u)=3$. Suppose $k>2$ and consider $G_1=G-v_1$. In every 3-colouring $\mathcal{C}_1$ of $G_1$ there is $\mathcal{C}_1(v_2) \neq \mathcal{C}_1(v_{2k-1})$, therefore by Theorem~\ref{thm:outer} in $P(G_1)$ there is a non-vanishing monomial $\eta u^2v_2^1v_{2k-1}^0 \Pi v_i^2$, which implies that $\eta u^3v_1^0v_2^2v_{2k-1}^0 \Pi v_i^2$ does not vanish in $P(G)$.
	
	Now let $G_2=G-\lbrace v_1,v_{2k-1} \rbrace$. Notice, that $G_2$ also satisfies conditions of the theorem, in particular that for every 3-colouring $\mathcal{C}_2$ of $G_2$ there is $\mathcal{C}_2(v_2) = \mathcal{C}_2(v_{2k-2})$. If there was a non-vanishing monomial $\eta u^4v_1^0v_2^1v_{2k-1}^0\Pi v_i^2$ in $P(G)$, then in $P(G_2)$ there would be a non-vanishing monomial $\eta u^2v_2^0v_{2k-2}^1\Pi v_i^2$, which contradicts Theorem~\ref{thm:outer}. Similarly, we rule out $\eta u^4v_1^0v_{2k-2}^1v_{2k-1}^0\Pi v_i^2$ from $P(G)$.
	
	Finally, fix $l \in \lbrace 2,\dots, 2k-3 \rbrace$. Observe that either $\mathcal{C}(v_1)=\mathcal{C}(v_l)$ and $\mathcal{C}(v_{l+1}) \neq \mathcal{C}(v_k)$ if $l$ is odd, or $\mathcal{C}(v_1) \neq \mathcal{C}(v_l)$ and $\mathcal{C}(v_{l+1}) = \mathcal{C}(v_k)$ if $l$ is even. Without loss of generality we can assume that $l$ is odd. Take $G-v_lv_{l+1}$ and split it into $G_1$ and $G_2$ on the cutvertex $u$. We have already established the $P(G_1)$ contains a non-vanishing monomial $\eta_1 u^3v_1^0v_l^0\Pi v_i^2$, and by Lemma~\ref{lem:012} $P(G_2)$ contains a non-vanishing monomial $\eta_2 u^1v_k^0\Pi v_i^2$, that we combine to obtain $\eta u^4v_1^0v_{2k-1}^0v_l^0\Pi v_i^2$. After we add back the edge $v_lv_{l+1}$, we have $\eta u^4v_1^0v_{2k-1}^0v_l^1\Pi v_i^2$ in $P(G)$. Had that monomial vanished in $P(G)$, $P(G-v_lv_{l+1})$ would have to contain a monomial $\eta u^4v_1^0v_{2k-1}^0v_l^1 v_{l+1}^1\Pi v_i^2$, but that would require a monomial $\eta_1' u^2v_1^0v_l^1\Pi v_i^2$ in $P(G_1)$, which is impossible by Theorem~\ref{thm:outer}. Suppose now that $\eta u^4v_1^0v_{2k-1}^0v_{l+1}^1\Pi v_i^2$ does not vanish in $P(G)$. However, that would require existence of either $\eta_1' u^2v_1^0v_l^1\Pi v_i^2$ in $P(G_1)$, which as we already know vanishes there, or $\eta_2' u^3v_1^0v_{2k-1}^0\Pi v_i^2$ in $P(G_2)$, which is contradicted by Lemma~\ref{lem:no_030_even}. Again, $|\eta|=|\eta_l|=1$ is a consequence of \cite[Theorem 2.1]{GT2021}.
\end{proof}
Theorem~\ref{thm:outer} guarantees that outerplanar near-triangulations with universal vertex of even order are ($2,2$)-$xy$-extendable (but not ($1,2$)-$xy$-extendable) with respect to the pair of degree 2 vertices. Theorem 2.1 from \cite{GT2021} shows even more, that one of the (degree 3) neighbours of either $x$ or $y$ also may be given a list of length 2. The following corollary shows that the possibility of such deficiency may be extended to a wider group of vertices.
\begin{corollary}\label{cor:all_121_odd}
	Let $G$ be a 2-connected outerplanar near-triangulation such that $$V(G)=\lbrace u,v_1,\dots,v_{2k-1} \rbrace, k>1,$$
	$$N_G(u)=\lbrace v_1,\dots,v_{2k-1} \rbrace,$$
	$$deg_G(v_1)=2=deg_G(v_{2k-1})$$
	and vertices $v_i$ are ordered counter-clockwise. Given $l$, $2 \le l <2k-1$, a monomial $\eta_l u^2v_1^1v_{2k-1}^1v_l^1\Pi v_i^2$ does not vanish if and only if~~$l$ is even. Moreover, if $\eta_l \neq 0$, then $|\eta_l|=1$.
\end{corollary}
\begin{proof}
	Let $G'$ be a graph constructed by adjoining $v_0$ to $u$ and $v_1$, and $v_{2k}$ to $u$ and $v_{2k-1}$. We notice that $\mathcal{C}'(v_0)=\mathcal{C}'(v_{2k})$ for every 3-colouring $\mathcal{C}'$ of $G'$, therefore Lemma~\ref{lem:odd_wheels} applies to $G'$. However, as the labelling of vertices $v_i$ in $G'$ starts from $v_0$, the parities of indexes of the vertices generating respective vanishing and non-vanishing monomials reverse in $P(G')$. Hence in $P(G')$ monomials $\eta_l u^2v_0^1v_{2k}^1v_l^1\Pi v_i^2$ do not vanish for every even $2 \le l < 2k$, and vanish for every odd $1 \le l < 2k+1$. We then remove $v_0$ and $v_{2k}$ from $G'$ to return to $G$ to justify the assertion.
\end{proof}

\section{Path 3-extendability and generalized wheels}\label{sec:wheels}

We are now in the right spot to define generalized wheels. As the term "generalized" may already suggest, this is not a single type of structure, but more an umbrella term. We will therefore define and investigate the components individually one by one, matching the definitions in~\cite{ThomassenExt}.
\begin{definition}
	An \emph{ordinary wheel} is a graph consisting of $k$-cycle on vertices $v_1,\dots ,v_k$, and a vertex $v$ neighboring all the vertices on the cycle. The path $\overrightarrow{v_kv_1v_2}$ will be called a \emph{principal path of $G$}, $v_1v_2$ and $v_kv_1$ will be its \emph{principal edges}, and $v_1,v_2,v_k$ will be its \emph{principal vertices}.
\end{definition}
Note that obviously any path of length 3 of an ordinary wheel may play the role of the principal path. The word "ordinary" is used here only to clearly distinguish that type of generalized wheel, as such a graph is usually referred to simply as \emph{wheel}, like in \cite{ThomassenExt}.

As evident from Theorem~\ref{thm:thomassenExt}, generalized wheels are crucial when looking for 3-path extendability in planar graphs. In this section, we will show that generalized wheels themselves are not 3-path-extendable. Before proceeding with subsequent definitions, we will prove a fact about 3-path extendability of ordinary wheels.
\begin{theorem}\label{thm:even_ok}
	Let $G$ be an ordinary wheel with principal path $v_kv_1v_2$ and inner vertex $v$. Then there is a nonvanishing monomial $\eta v_1^0v_2^0v_k^0v^\alpha \Pi v_i^{\alpha_i}$ in $P(G-\overrightarrow{v_kv_1v_2})$, with $\alpha_i \le 2$ and $\alpha \le 4$, if and only if $k$ is even.
\end{theorem}
\begin{proof}
	As
	$$P(G-\overrightarrow{v_kv_1v_2})=P(G-v_1)(v-v_1),$$
	we can equivalently show that in $P(G_1-v_1)$ there is a non-vanishing monomial $\eta v^\alpha v_2^0v_k^0\Pi v_i^{\alpha_i}$ with $\alpha_i \le 2$ and $\alpha \le 3$. If $k$ is even, then $G-v_1$ is an outerplanar near-triangulation with even number of vertices and universal vertex $v$. Hence, by Lemma~\ref{lem:odd_wheels}, $P(G-v_1)$ contains a non-vanishing monomial $\eta v^3v_2^0v_k^0\Pi v_i^2$. Conversely, if $k$ is odd, then $G-v_1$ is an outerplanar near-triangulation with odd number of vertices and universal vertex $v$. By Lemma~\ref{lem:no_030_even}, the monomial $v^3v_2^0v_k^0\Pi v_i^2$ vanishes in $P(G-v_1)$, as do any monomial of the form $v^\alpha v_2^0v_k^0\Pi v_i^{\alpha_i}$ with $\alpha, \alpha_i \le 2$,  by Theorem~\ref{thm:outer}. It is therefore impossible for $P(G-\overrightarrow{v_kv_1v_2})$ to contain the monomial from the assertion.
\end{proof}	
Consequently, we see that not all ordinary wheels are problematic --- only the ones that have an odd number of vertices on the outer cycle. We will call such graphs \emph{odd wheels}, while \emph{even wheels} will denote ordinary wheels with the outer cycle of even length.

\begin{figure}[htb]
	\begin{center}
		\begin{minipage}{0.45\linewidth}
			\includegraphics[width = 0.9\textwidth]{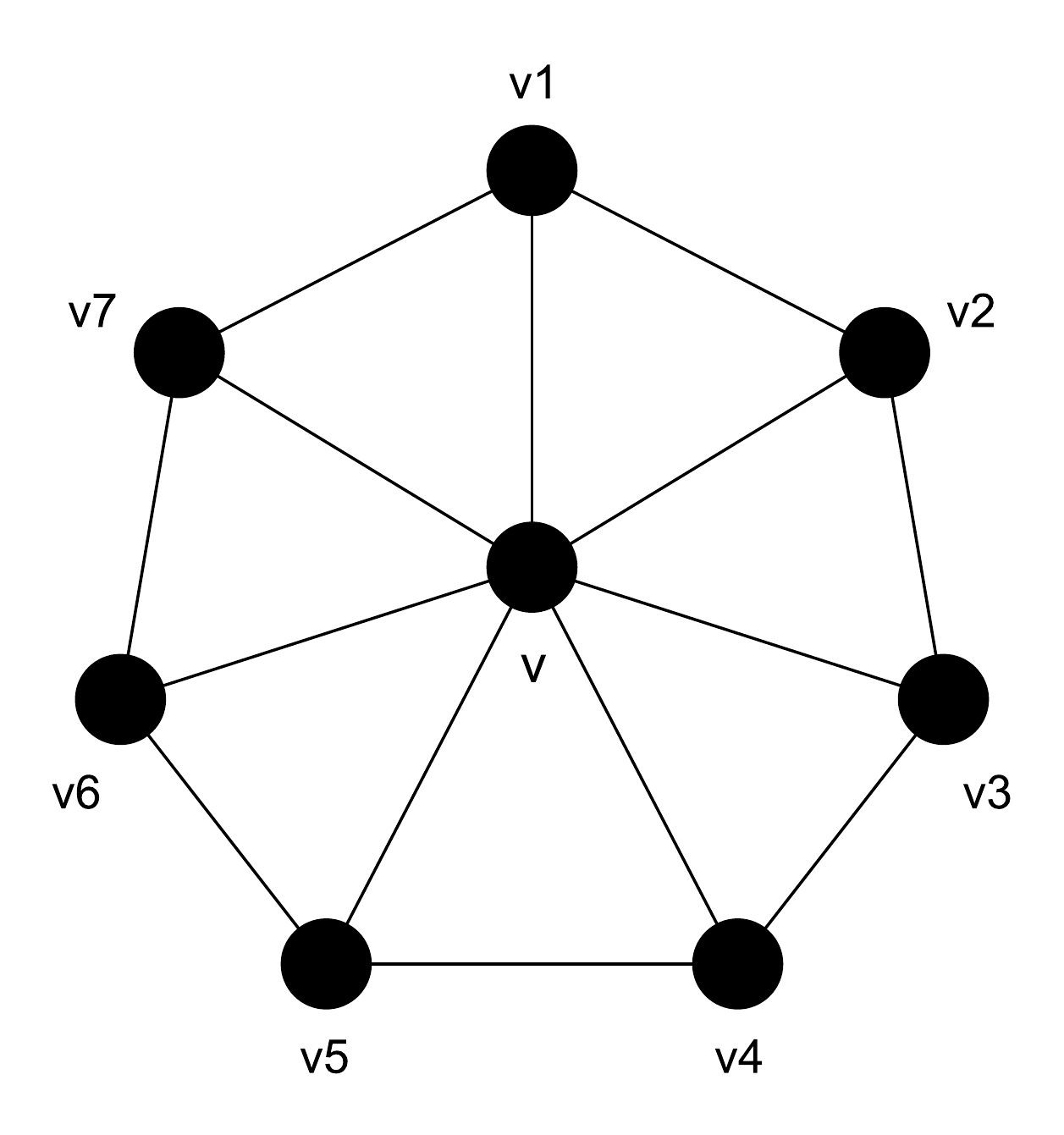}
		\end{minipage}
		\begin{minipage}{0.45\linewidth}
			\includegraphics[width = 0.9\textwidth]{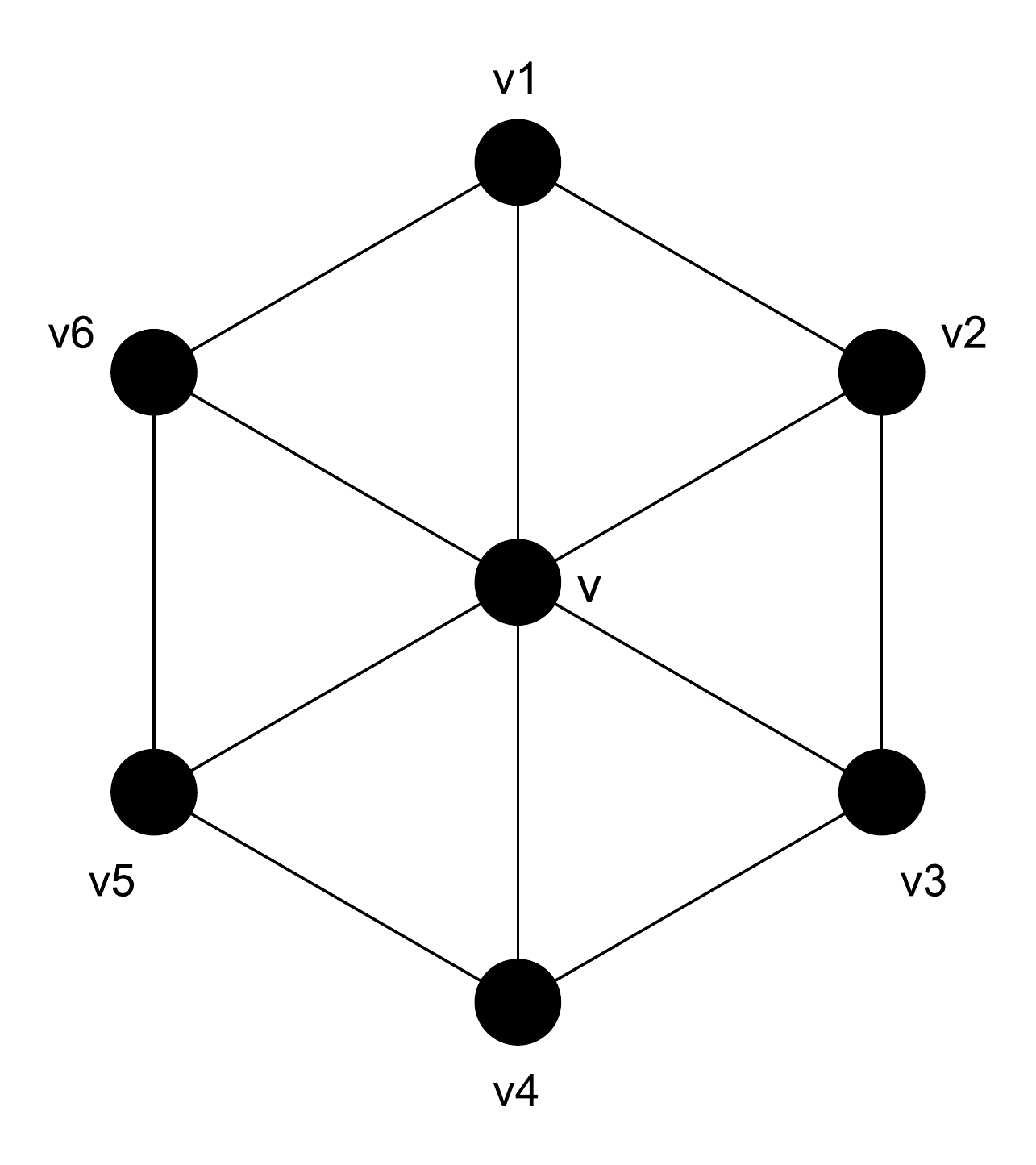}
		\end{minipage}
	\end{center}
	\caption{Examples of ordinary wheels. Left - an odd wheel, right - an even wheel.} \label{fig:odd_even}
\end{figure}

Taking a look on the assertion of Theorem~\ref{thm:even_ok}, one can note the particular approach to the polynomial understanding of 3-path extendability --- namely the exclusion of the principal edges and requiring the exponent equal to 0 on all vertices of the principal path. This may be understood as assigning the list of length 1 to these vertices with the restriction that lists of neighboring vertices cannot be identical. It is also immediate that the existence of the monomial of the form $Av_1^0v_2^0v_k^0$ in $P(G-\overrightarrow{v_kv_1v_2})$ directly implies the existence of all of the following:
\begin{itemize}
	\item $Av_1^0v_2^1v_k^0$ in $P(G-\overrightarrow{v_kv_1})$;
	\item $Av_1^0v_2^0v_k^1$ in $P(G-\overrightarrow{v_1v_2})$;
	\item $Av_1^0v_2^1v_k^1$ in $P(G)$.
\end{itemize}
In fact, these are all equivalent.

With all this cleared, we will continue the task of defining generalized wheels.
\begin{definition}
	A \emph{broken wheel} is a graph consisting of the path $v_2 \dots v_k$ and a vertex $v_1$ neighboring all the vertices $v_2,\dots, v_k$. The path $\overrightarrow{v_kv_1v_2}$ will be called a \emph{principal path of $G$}, $v_1v_2$ and $v_kv_1$ will be its \emph{principal edges}, and $v_1,v_2,v_k$ will be its \emph{principal vertices}.
\end{definition}
We can therefore see, that broken wheels and outerplanar near-triangulations with universal vertex are \emph{the same exact class of graphs}, which brings the much needed context to the contents of Section~\ref{sec:sec3}. An example of such broken wheel can be therefore seen on Figure~\ref{fig:univ_vertex}. Observe also that these wheels are called "broken", because such a graph may arise from deletion of a single edge from the outer cycle of an ordinary wheel. Contrary to the case of ordinary wheels, the principal path of a broken wheel is in most cases unique. We will now show that no broken wheel can be 3-path extendable.
\begin{theorem}
	Let $G$ be a 2-connected outerplanar near-triangulation on $k$ vertices, and let $v_k, v_1, v_2$ form a path on the outer cycle of $G$. Then in $P(G - \overrightarrow{v_kv_1v_2})$ there is no monomial $v_1^0v_2^0v_k^0\Pi v_i^{\alpha_i}$ with $\alpha_i \le 2$.
\end{theorem}
\begin{proof}
	It suffices to notice that the sum of exponents in desired monomial would be no larger than $$2(k-3)=2k-6,$$ while $G-\overrightarrow{v_kv_1v_2}$ has $2k-5$ edges.
\end{proof}
\begin{corollary}\label{cor:no_000_broken}
	Let $G$ be a broken wheel with principal path $\overrightarrow{v_kv_1v_2}$. Then in $P(G-\overrightarrow{v_kv_1v_2})$ there is no monomial $\eta v_1^0v_2^0v_k^0 \Pi v_i^{\alpha_i}$, with $\alpha_i \le 2$.
\end{corollary}
Te only step left in the quest of defining generalized wheels is the final type, the \emph{multiple wheels}, which as name suggest will be a "compound" type. We will define it iteratively, starting from a \emph{double wheel}.
\begin{definition}
	Let each of $G_1$ and $G_2$ be either an ordinary wheel or a broken wheel. Let $G$ be constructed by identifying vertices $v_1$ together with one of the principal edges from each of these graphs. Graph $G$ is then called a \emph{double wheel}, and $G_1$ and $G_2$ are \emph{components} of $G$. This process can be iterated, and every graph resulting from that process will be called a \emph{multiple wheel}. The vertex that all the vertices $v_1$ from the components were combined into will be labelled $v_1$, and the principal vertices that were not identified with others at any point of the process will be labelled $v_2$ and $v_k$. The path $\overrightarrow{v_kv_1v_2}$ will be called a \emph{principal path of $G$}, $v_1v_2$ and $v_kv_1$ will be its \emph{principal edges}, and $v_1,v_2,v_k$ will be its \emph{principal vertices}.
\end{definition}

\begin{figure}[htb]
	\begin{center}
		\includegraphics[width = 0.75\textwidth]{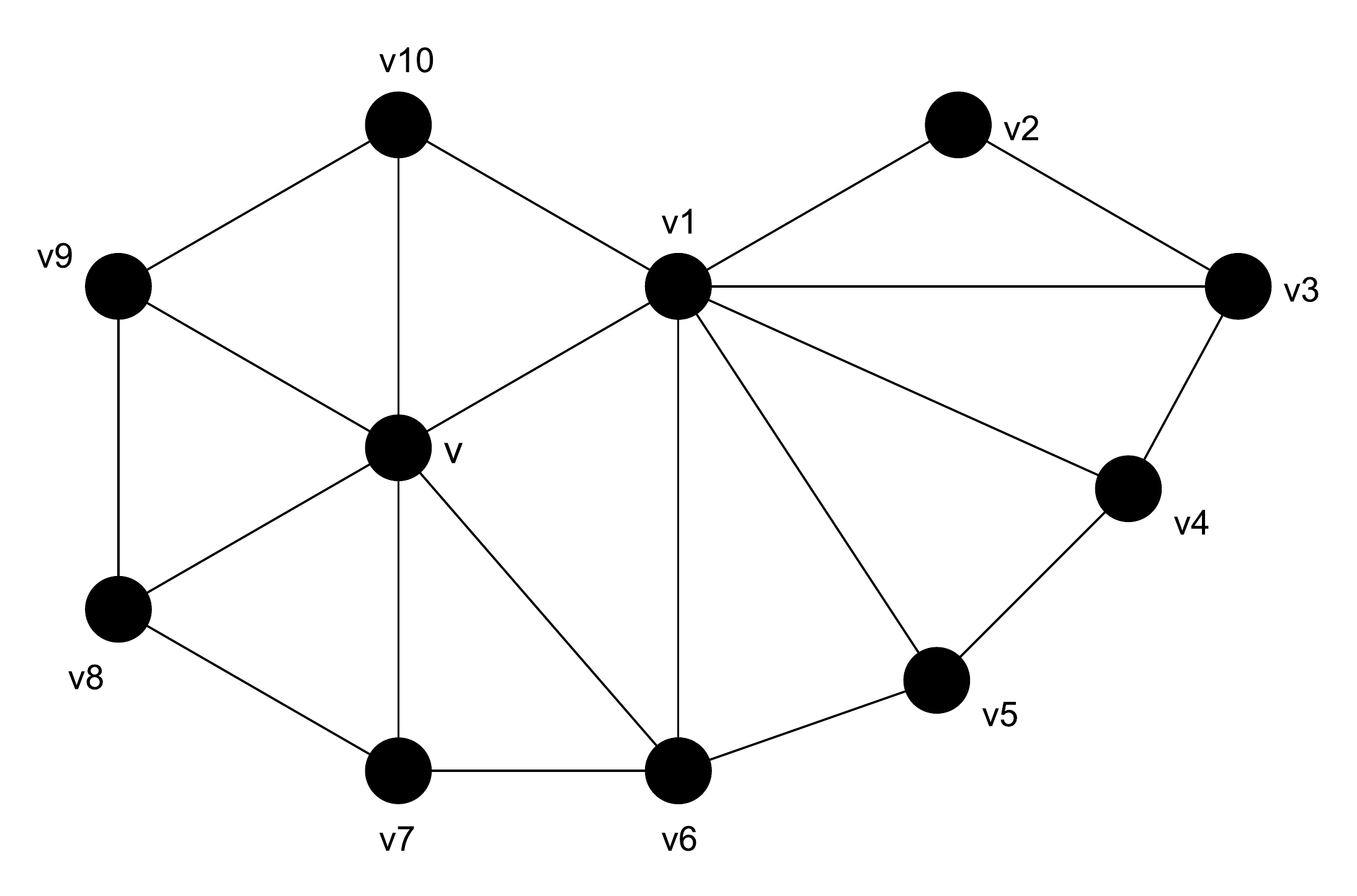}
		
		\includegraphics[width = 0.75\textwidth]{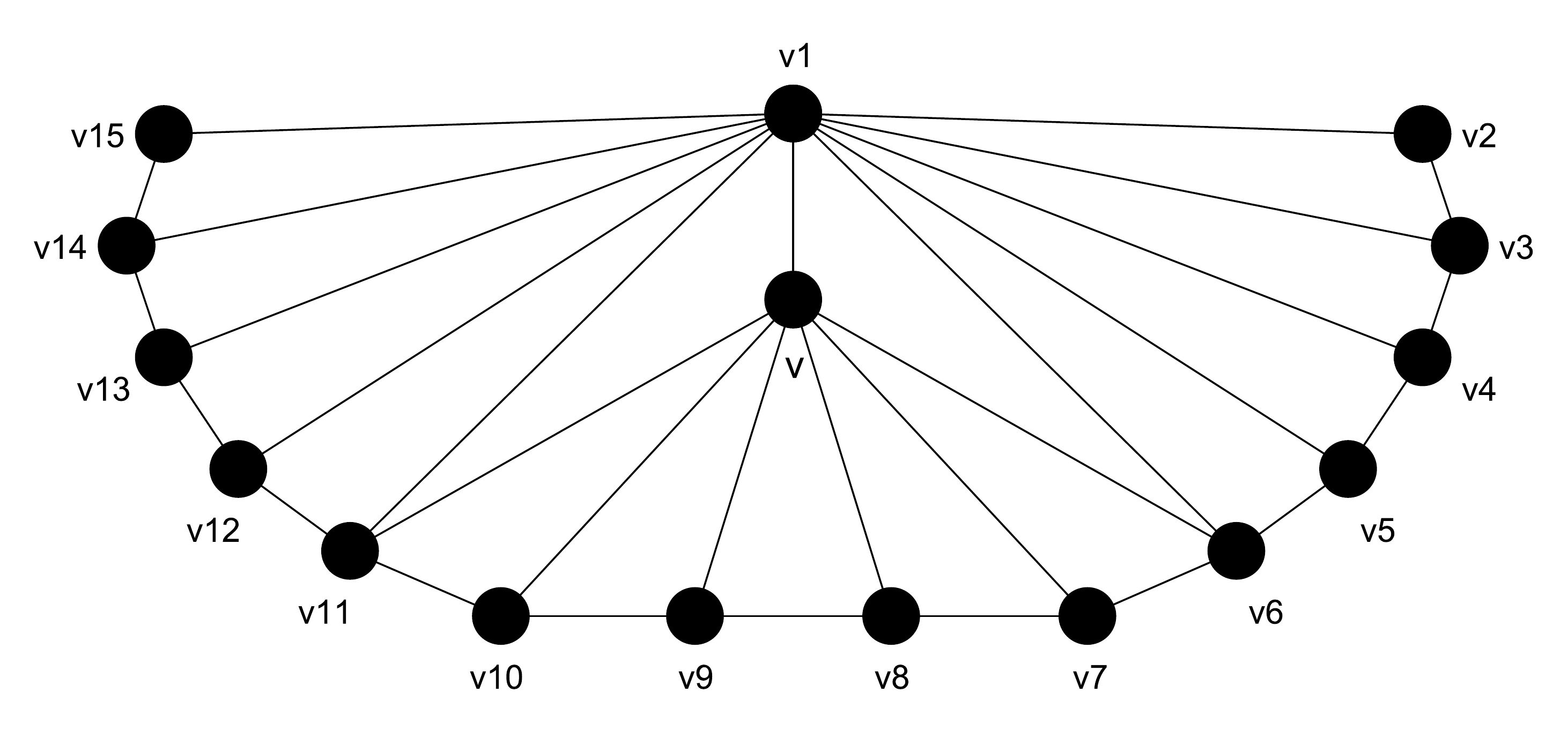}
	\end{center}
	\caption{Examples of multiple wheels constructed from wheels of Figures~\ref{fig:univ_vertex} and~\ref{fig:odd_even}. Top: a double wheel with an even wheel component. Edge $v_1v_6$ is an identification of edges $v_1v_2$ and $v_6v_1$ from the odrinary and broken wheel, respectively. The principal path is $v_{10}v_1v_2$; Bottom: a multiple wheel without an even wheel component. Edges $v_1v_6$ and $v_1v_{11}$ are a result of identifiaction of edges from the components, $v_{15}v_1v_2$ is the principal path.} \label{fig:multiple}
\end{figure}

Looking at the definition of multiple wheels, it appears that it doesn't make much sense to construct multiple wheel solely of broken wheels --- such a multiple wheel will still simply be a broken wheel. We will therefore assume that at least one of the components of the multiple wheel is an ordinary wheel. We are now finally ready to define generalized wheels.
\begin{definition}
	A graph is a \emph{generalized wheel} if it is a broken wheel, an odd wheel or a multiple wheel consisting of broken and odd wheels, without an even wheel component. The notions of principal paths, edges and vertices carry over from the previous definitions.
\end{definition}
The exclusion of even wheels from the definition of generalized wheels may not immediately seem justified --- for example Thomassen in~\cite{ThomassenExt} does not even distinguish between even and odd wheels. However, the generalized wheels are introduced as the "forbidden configurations" for 3-path extendability, and they are intended not to be 3-path extendable. But we have already shown that the even wheels are in fact 3-path extandable, and now we will prove that a multiple wheel with an even wheel component is also 3-path extendable, while a multiple wheel without such component is not.
\begin{theorem}\label{thm:even_multi}
	Let $G$ be a multiple wheel with principal path $\overrightarrow{v_kv_1v_2}$ and an even wheel component $G_0$, where $v_3,\dots ,v_{k-1}$ are the remaining vertices on the outer cycle of $G$ and $u_1,\dots,u_m$ are the vertices of $G$ not on the outer cycle. Then $P(G-\overrightarrow{v_kv_1v_2})$ contains a non-vanishing monomial $\eta v_1^0v_2^0v_k^0\Pi v_i^{\alpha_i}\Pi u_j^{\beta_j}$, with $\alpha_i \le 2$ and $\beta_j \le 4$.
\end{theorem}
\begin{proof}
	Let $G'$ be a graph resulting from the deletion of all non-principal vertices of $G_0$ from $G$. The vertex $v_1$ is a cutvertex in $G'$, therefore by Theorem~\ref{thm:zhu} and Lemma~\ref{lem:l2} $P(G'-\overrightarrow{v_kv_1v_2})$ contains a non-vanishing monomial $\eta' v_1^0v_2^0v_k^0\Pi v_i^{\alpha_i}\Pi u_j^{\beta_j}$, with $\alpha_i \le 2$ and $\beta_j \le 4$. In case $v_2$ or $v_k$ is a principal vertex of $G_0$, it is a leaf in $G'$ and an isolated vertex in $G'-\overrightarrow{v_kv_1v_2}$. This however is not an obstacle in the argument. Now, by Theorem~\ref{thm:even_ok}, $G_0$ has a monomial $\eta'' v_1^0v_p^0v_q^0\Pi v_i^2 u_r^4$ in $P(G_0 - \overrightarrow{v_qv_1v_p})$, and as 
	$$P(G-\overrightarrow{v_kv_1v_2})=P(G'-\overrightarrow{v_kv_1v_2})P(G_0-\overrightarrow{v_qv_1v_p}),$$
	this combines with the monomial mentioned earlier into the required monomial in $P(G-\overrightarrow{v_kv_1v_2})$.
\end{proof}
\begin{theorem}
	Let $G$ be a double wheel with principal path $\overrightarrow{v_kv_1v_2}$ with components $G_1$ and $G_2$, none of which is an even wheel, where $v_3,\dots ,v_{k-1}$ are the remaining vertices on the outer cycle of $G$ and $u_1,\dots,u_m$ are the vertices of $G$ not on the outer cycle. Then every monomial $v_1^0v_2^0v_k^0\Pi v_i^{\alpha_i}\Pi u_j^{\beta_j}$, with $\alpha_i \le 2$ and $\beta_j \le 4$ vanishes in $P(G-\overrightarrow{v_kv_1v_2})$.
\end{theorem}
\begin{proof}
	Suppose otherwise. Let $v_p$ be the vertex of $G$ that the principal vertices of $G_1$ and $G_2$ were identified into, and let $e=v_1v_p$. By the hypothesis, $P(G-\overrightarrow{v_kv_1v_2})$ contains a non-vanishing monomial $\eta v_1^0v_2^0v_k^0\Pi v_i^{\alpha_i}\Pi u_j^{\beta_j}$, with $\alpha_i \le 2$ and $\beta_j \le 4$. Without loss of generality we may assume that $\alpha_p=2$ in that monomial. Hence $P(G-\overrightarrow{v_kv_1v_2}-e)$ contains a monomial $\eta v_1^0v_2^0v_k^0v_p^1\Pi v_i^{\alpha_i}\Pi u_j^{\beta_j}$, with $\alpha_i \le 2$ and $\beta_j \le 4$. As
	$$P(G-\overrightarrow{v_kv_1v_2}-e)=P(G_1-\overrightarrow{v_kv_1v_p})P(G_2-\overrightarrow{v_pv_1v_2}),$$
	where vertices of $G_1$ and $G_2$ are appropriately relabelled, we see that to produce such a monomial either $P(G_1-\overrightarrow{v_kv_1v_p})$ would have to contain a non-vanishing monomial $\eta v_1^0v_2^0v_p^0\Pi v_i^{\alpha_i}\Pi u_j^{\beta_j}$, or $P(G_1-\overrightarrow{v_kv_1v_p})$ would have to contain a non-vanishing monomial $\eta v_1^0v_p^0v_k^0\Pi v_i^{\alpha_i}\Pi u_j^{\beta_j}$, with $\alpha_i \le 2$ and $\beta_j \le 4$. This is however contradicted by Theorem~\ref{thm:even_ok} and Corollary~\ref{cor:no_000_broken}.
\end{proof}
Inductively, we obtain the following:
\begin{corollary}
	Let $G$ be a multiple wheel with principal path $\overrightarrow{v_kv_1v_2}$ without an even wheel component, where $v_3,\dots ,v_{k-1}$ are the remaining vertices on the outer cycle of $G$ and $u_1,\dots,u_m$ are the vertices of $G$ not on the outer cycle. Then every monomial $v_1^0v_2^0v_k^0\Pi v_i^{\alpha_i}\Pi u_j^{\beta_j}$, with $\alpha_i \le 2$ and $\beta_j \le 4$ vanishes in $P(G-\overrightarrow{v_kv_1v_2})$.
\end{corollary}
As it was in the case of the outerplanar near-triangulations with universal vertex in Section~\ref{sec:sec3}, the information about 3-path extendability of generalized wheels is an insufficient tool to investigate the 3-path extendability of planar graphs. Therefore we need to observe and prove several facts about their polynomials.
\begin{observation} \label{lem:broken5}
	Let $G$ be a broken wheel on 5 vertices $V(G) = \lbrace v_1, \dots, v_5 \rbrace$, where $deg(v_1)=4$ and $deg(v_2)=deg(v_5)=2$. Then
	$$P(G-\overrightarrow{v_5v_1v_2}) = v_3^2v_4^2(v_5^1-v_2^1) + v_3^2v_4^1(v_1^2+v_1^1v_2^1+v_2^1v_5^1) - v_3^1v_4^2(v_1^2+v_1^1v_5^1+v_2^1v_5^1) + Q(G-\overrightarrow{v_5v_1v_2}),$$
	where polynomial $Q(G-\overrightarrow{v_5v_1v_2})$ consists of the monomials $v_1^{\alpha_1}v_2^{\alpha_2}v_3^{\alpha_3}v_4^{\alpha_4}v_5^{\alpha_5}$such that either $\alpha_3 > 2, \alpha_4 > 2$ or $\alpha_1+\alpha_2+\alpha_5>2$.
\end{observation}
\begin{lemma} \label{lem:ordinary2k1}
	Let $G$ be an odd wheel with $2k+1$ vertices on the outer cycle, say $v_1, v_2, \dots, v_{2k+1}$, and inner vertex $u$. Then
	$$P(G-\overrightarrow{v_{2k+1}v_1v_2}) = v_3^2 v_4^2 \dots v_{2k}^2 u^3 v_1^0 (v_2 - v_{2k+1}) +$$
	$$ + v_2^0 (v_1 + v_{2k+1}) u^4 Q_1(G-\overrightarrow{v_{2k+1}v_1v_2}) - v_{2k+1}^0 (v_1 + v_2)u^4 Q_2(G-\overrightarrow{v_{2k+1}v_1v_2}) + $$
	$$ \pm v_3^0 v_4^2 \dots v_{2k}^2 u^4 v_1^1 v_2^1 v_{2k+1}^0 \pm v_3^2 v_4^2 \dots v_{2k-1}^2 v_{2k}^0 u^4 v_1^1 v_2^0 v_{2k+1}^1+Q(G-\overrightarrow{v_{2k+1}v_1v_2}),$$
	where
	$$Q_1(G-\overrightarrow{v_{2k+1}v_1v_2})=\sum_{2|l} \Pi v_i^2v_l^1, 3 \le i \le 2k,$$
	$$Q_2(G-\overrightarrow{v_{2k+1}v_1v_2})=\sum_{2 \nmid l} \Pi v_i^2v_l^1, 3 \le i \le 2k$$
	and $Q(G-\overrightarrow{v_{2k+1}v_1v_2})$ is the remaining part. Moreover, any monomial $\eta u^\beta \prod v_i^{\alpha_i}$ in $Q(G-\overrightarrow{v_{2k+1}v_1v_2})$ such that $\alpha_1+\alpha_2+\alpha_{2k+1} \le 2$ and $\beta \le 4$ is of the form
	$$\eta u^4v_1^1v_2^{\alpha_2}v_{2k+1}^{\alpha_{2k+1}}v_l^0 \prod_{i \neq l}v_i^2,$$
	with $\alpha_2+\alpha_{2k+1}=1$ and $4 \le l \le 2k-1$.
\end{lemma}
\begin{proof}
	Observe, that $G-v_1$ is the outerplanar near triangulation of the form:
	$$V(G-v_1)=\lbrace u,v_2,\dots,v_{2k+1}\rbrace,$$
	$$N_{G-v_1}(u)=\lbrace v_2,\dots,v_{2k+1}\rbrace,$$
	$$deg_{G-v_1}(v_2)=deg_{G-v_1}(v_{2k+1})=2,$$
	$$\mathcal{C}(v_2) \neq \mathcal{C}(v_{2k+1}),$$
	where $\mathcal{C}$ is any 3-colouring of $G-v_1$.
	
	Take a look at first at the $v_3^2 v_4^2 \dots v_{2k}^2 u^3 v_1^0 (v_2 - v_{2k+1})$ part of $P(G-\overrightarrow{v_{2k+1}v_1v_2})$. Its existence is equivalent to the existence of $-v_2^1v_3^2 v_4^2 \dots v_{2k}^2v_{2k+1}^0 u^2$ and $v_2^0v_3^2 v_4^2 \dots v_{2k}^2 v_{2k+1}^1u^2$ in $P(G-v_1)$. But as $\mathcal{C}(v_2) \neq \mathcal{C}(v_{2k+1})$ in every 3-colouring $\mathcal{C}$ of $G-v_1$, their existence is guaranteed by Theorem~\ref{thm:outer}, and the fact that their signs differ follows from Lemma~\ref{lem:signsymmetry}. On the other hand, the existence of $u^3v_1^1v_2^0v_3^2 v_4^2 \dots v_{2k}^2v_{2k+1}^0$ in $P(G-\overrightarrow{v_{2k+1}v_1v_2})$ would also imply existence of $u^4v_1^0v_2^0v_3^2 v_4^2 \dots v_{2k}^2v_{2k+1}^0$ there, which is not the case by Theorem~\ref{thm:even_ok}.
	
	Now consider the part
	$$v_2^0 (v_1 + v_{2k+1}) u^4 Q_1(G-\overrightarrow{v_{2k+1}v_1v_2}) - v_{2k+1}^0 (v_1 + v_2)u^4 Q_2(G-\overrightarrow{v_{2k+1}v_1v_2}),$$
	and rewrite it as
	$$v_1^1v_2^0v_{2k+1}^0u^4(Q_1(G-\overrightarrow{v_{2k+1}v_1v_2})-Q_2(G-\overrightarrow{v_{2k+1}v_1v_2}))+$$ 
	$$+v_1^0v_2^0v_{2k+1}^1u^4Q_1(G-\overrightarrow{v_{2k+1}v_1v_2})-v_1^0v_2^1v_{2k+1}^0u^4Q_2(G-\overrightarrow{v_{2k+1}v_1v_2}).$$
	We will therefore examine $P(G-v_1)$ for the existence of
	$$v_2^0v_{2k+1}^0u^4(Q_1(G-\overrightarrow{v_{2k+1}v_1v_2})-Q_2(G-\overrightarrow{v_{2k+1}v_1v_2}))-$$ 
	$$-v_2^0v_{2k+1}^1u^3Q_1(G-\overrightarrow{v_{2k+1}v_1v_2})+v_2^1v_{2k+1}^0u^3Q_2(G-\overrightarrow{v_{2k+1}v_1v_2}).$$
	The first component can be easily deduced as a consequence of Lemma~\ref{lem:even_wheels}. For the second one, consider $G-\lbrace v_1,v_2 \rbrace$. As $\mathcal{C}(v_3)=\mathcal{C}(v_{2k+1})$, we may apply Corollary~\ref{cor:all_121_odd} to find monomials of the form $v_3^1v_{2k+1}^1u^2v_l^1\Pi v_i^2$ for every even $l$ in $P(G-\lbrace v_1,v_2 \rbrace)$, which imply the existence of all necessary monomials in $P(G-v_1)$. Symmetrically, we find all the monomials from the third component by analyzing $G-\lbrace v_1,v_{2k+1} \rbrace$. The opposite signs are again inferred from Lemma~\ref{lem:signsymmetry}. Finally, from the presence of equivalence in the statement of Corollary~\ref{cor:all_121_odd} we deduce that
	$$\pm v_1^0v_2^1v_{2k+1}^0 u^4 Q_1(G-\overrightarrow{v_{2k+1}v_1v_2}) \pm v_1^0v_2^0v_{2k+1}^1 u^4 Q_2(G-\overrightarrow{v_{2k+1}v_1v_2})$$
	vanishes in $P(G-\overrightarrow{v_{2k+1}v_1v_2})$.
	
	We now move to the last two monomials from the assertion. We will look for
	$$\pm v_2^1 v_3^0 v_4^2 \dots v_{2k}^2 u^4 v_{2k+1}^0 \pm v_2^0v_3^2 v_4^2 \dots v_{2k-1}^2 v_{2k}^0 u^4 v_{2k+1}^1$$
	in $P(G-v_1)$. By Lemma~\ref{lem:odd_wheels} we can find the monomial  $\pm v_3^0 v_4^2 \dots v_{2k}^2 u^3 v_{2k+1}^0$ in $P(G-\lbrace v_1,v_2 \rbrace)$, and as
	$$P(G-v_1)=P(G-\lbrace v_1,v_2 \rbrace)(v_2-v_3)(v_2-u)=P(G-\lbrace v_1,v_2 \rbrace)(v_2^2-v_2^1u^1-v_2^1v_3^1+v_3^1u^1),$$
	$P(G-v_1)$ contains a non-vanishing $\pm v_2^1 v_3^0 v_4^2 \dots v_{2k}^2 u^4 v_{2k+1}^0$. As in the previous case, we find the second monomial with the symmetric argument concerning $G-\lbrace v_1,v_{2k+1} \rbrace$.
	
	It remains to show that monomials $\eta v_3^0 v_4^2 \dots v_{2k}^2 u^4 v_1^1 v_2^0 v_{2k+1}^1$, $\eta v_3^0 v_4^2 \dots v_{2k}^2 u^4 v_1^0 v_2^1 v_{2k+1}^1$, $\eta v_3^2 v_4^2 \dots v_{2k-1}^2 v_{2k}^0 u^4 v_1^0 v_2^1 v_{2k+1}^1$ and $\eta v_3^2 v_4^2 \dots v_{2k-1}^2 v_{2k}^0 u^4 v_1^1 v_2^1 v_{2k+1}^0$ all vanish in $P(G-\overrightarrow{v_{2k+1}v_1v_2})$. The first and the last one are obviously impossible, as in each of them there is a pair of neighbouring vertices that both have exponent equal to 0. Taking on the second one, it would imply that there was $\eta v_3^0 v_4^2 \dots v_{2k}^2 u^2 v_{2k+1}^1$ in $P(G-\lbrace v_1,v_2 \rbrace)$, but as $\mathcal{C}(v_3)=\mathcal{C}(v_{2k+1})$, this would contradict Theorem~\ref{thm:outer}. Similarly, the fact that $\mathcal{C}(v_2)=\mathcal{C}(v_{2k})$ rules out the third of the monomials.
\end{proof}
We know that generalized wheels are not 3-path extendable. It turns out, however, that they are not far from it, as deletion of the single interior edge from such a wheel is enough to ensure the 3-path extendability.
\begin{lemma}
	Let $G$ be an odd wheel with principal path $v_kv_1v_2$, $k>3$, and inner vertex $v$, and let $e$ be any edge of $G$ incident to $v$. Then there is a non-vanishing monomial $\eta v_1^0v_2^0v_k^0v^\alpha \Pi v_i^{\alpha_i}$ with $\alpha_i \le 2, \alpha \le 4$ in $P(G-\overrightarrow{v_kv_1v_2}-e)$.
\end{lemma}
\begin{proof}
	Let
	$$G'=G-\overrightarrow{v_kv_1v_2}-e.$$
	We have four possible cases:
	
	\textit{Case 1: $e=vv_1$.} In this case, $v_1$ is an isolated vertex in $G'$, and $G'-v_1$ is an outerplanar near-triangulation with odd number of vertices and universal vertex $v$. By Lemma~\ref{lem:all_040_even} there is a non-vanishing monomial $\eta v^4v_2^0v_k^0\Pi v_i^{\alpha_i}, \alpha_i \le 2$ in $P(G'-v_1)$, and as $v_1$ is isolated, there is a non-vanishing monomial $\eta v^4v_1^0v_2^0v_k^0\Pi v_i^{\alpha_i}, \alpha_i \le 2$ in $P(G')$.
	
	\textit{Case 2: $e=vv_2$ or $e=vv_k$.} Take $e=vv_2$ first. One can see that
	$$G''=G'-\lbrace v_1,v_2 \rbrace$$
	is an outerplanar near-triangulation with even number of vertices and universal vertex $v$. By Lemma~\ref{lem:odd_wheels} $P(G'')$ contains a non-vanishing monomial $\eta v^3v_k^0v_3^0\Pi v_i^2$. Since
	$$N_{G'}(v_1)=\lbrace v \rbrace$$
	and
	$$N_{G'}(v_2)=\lbrace v_3 \rbrace,$$
	we have that
	$$P(G')=P(G'')(v_1-v)(v_2-v_3),$$
	therefore $P(G')$ contains a non-vanishing monomial $\eta v^4v_1^0v_2^0v_k^0v_3^1\Pi v_i^2$. If $e=vv_k$, the argument is symmetric.
	
	\textit{Case 3: $e=vv_3$ or $e=vv_{k-1}$.} Take $e=vv_3$ first. Here
	$$G''=G'-\lbrace v_1,v_2,v_3 \rbrace$$
	is an outerplanar near-triangulation with universal vertex $v$. By Lemma~\ref{lem:012} $P(G'')$ contains a non-vanishing monomial $-\eta v^1v_k^0\Pi v_i^2$. Since
	$$N_{G'-v_3}(v_1)=N_{G'-v_3}(v_2)=\lbrace v \rbrace,$$
	we have that
	$$P(G'-v_3)=P(G'')(v_1-v)(v_2-v),$$
	therefore $P(G'-v_3)$ contains a non-vanishing monomial $-\eta v^3v_1^0v_2^0v_k^0\Pi v_i^2$. Finally, $N_{G'}(v_3)=\lbrace v_2,v_4 \rbrace$, therefore
	$$P(G')=P(G'-v_3)(v_2-v_3)(v_3-v_4)$$
	contains a non-vanishing monomial $\eta v^3v_1^0v_2^0v_k^0v_3^2\Pi v_i^2$. If $e=vv_{k-1}$, the argument is symmetric.
	
	\textit{Case 4: $e=vv_p, p \notin \lbrace 1,2,3,k-1,k \rbrace$.} For this case to be different than previous ones, we need $k>5$. Let
	$$G''=G'-\lbrace v_1,v_p \rbrace.$$
	The graph $G''$ has a cutvertex $v$, so we split it by this cutvertex into graphs $G_1$ and $G_2$, with $v_2 \in V(G_1)$. These graphs are both outerplanar near-triangulations with universal vertex $v$, so $P(G_1)$ and $P(G_2)$ contain non-vanishing monomials $\eta_1 v^1v_2^0\Pi v_i^2$ and $\eta_1 v^1v_k^0\Pi v_i^2$, respectively. By Lemma~\ref{lem:l2} there is a non-vanishing monomial $\eta v^2v_2^0v_k^0\Pi v_i^2$ in $P(G'')$, where $\eta=\eta_1\eta_2$. Since
	$$N_{G'}(v_1)=\lbrace v \rbrace$$
	and
	$$N_{G'}(v_p)=\lbrace v_{p-1},v_{p+1} \rbrace,$$
	we have that
	$$P(G')=P(G'')(v_1-v)(v_{p-1}-v_p)(v_p-v_{p+1}),$$
	therefore $P(G')$ contains a non-vanishing monomial $\eta v^3v_1^0v_2^0v_k^0v_p^2\Pi v_i^2$.
\end{proof}
\begin{lemma}
	Let $G$ be a broken wheel with principal path $\overrightarrow{v_kv_1v_2},k>3$, and let $e=v_1v_p, p \notin \lbrace 2,k \rbrace$. Then in $P(G-\overrightarrow{v_kv_1v_2}-e)$ there is a non-vanishing monomial $\eta v_1^0v_2^0v_k^0\Pi v_i^2$.
\end{lemma}
\begin{proof}
	Let $G'=G-\overrightarrow{v_kv_1v_2}-e$, and start with a path $\overrightarrow{v_2\dots v_k}$. By Lemma~\ref{lem:l2}, the polynomial of that path contains a non-vanishing monomial $\pm v_2^0v_k^0v_p^2\Pi v_i^1$. Now
	$$P(G')=P(\overrightarrow{v_2\dots v_k})\prod_{i \neq p}(v_1-v_i),$$
	hence $P(G')$ contains a non-vanishing monomial $\eta v_1^0v_2^0v_k^0v_p^2\Pi v_i^2$.
\end{proof}
From these lemmas about odd and broken wheels, with the help of Theorem~\ref{thm:zhu}, we obtain a corollary for all generalized wheels. 
\begin{corollary}\label{cor:gen-e}
	Let $G$ be a generalized wheel with a principal path $v_1v_2v_k$, $k>3$, without a component that is an odd wheel with three vertices on the outer cycle, and let $e$ be an edge of $G$ not on the outer cycle. Then in $P(G-\overrightarrow{v_kv_1v_2}-e)$ there is a non-vanishing monomial $\eta v_1^0v_2^0v_k^0\Pi v_i^{\alpha_i}\Pi u_j^{\beta_j}, \alpha_i \le 2, \beta_j \le 4$.
\end{corollary}
\begin{proof}
	If $e$ is an interior edge of any of the components, we may use a reasoning akin to the proof of Theorem~\ref{thm:even_multi}. If $e$ is one of the identified principal edges, let $v_p$ be the endpoint of that edge other than $v_1$. Edge $e=v_1v_p$ is a chord in $G-\overrightarrow{v_kv_1v_2}$, so we use this edge to split $G-\overrightarrow{v_kv_1v_2}$ into two graphs. In one of these graphs $e$ will still be present, so we remove it. The resulting graphs are $G_1$ and $G_2$, with $v_2 \in V(G_2)$. By Theorem~\ref{thm:zhu}, $P(G_1)$ and $P(G_2)$ contain monomials $\eta_1 v_1^0v_k^0v_p^1\Pi v_i^{\alpha_i}\Pi u_j^{\beta_j}$ and $\eta_2 v_1^0v_2^0v_p^1\Pi v_i^{\alpha_i}\Pi u_j^{\beta_j}$ with $\alpha_i \le 2, \beta_j \le 4$, respectively, so these combine into $\eta v_1^0v_2^0v_k^0v_p^2\Pi v_i^{\alpha_i}\Pi u_j^{\beta_j}, \alpha_i \le 2, \beta_j \le 4$ in $P(G-\overrightarrow{v_kv_1v_2}-e)$, and $\eta \neq 0$ as there was no monomial $\eta'_1 v_1^0v_k^0v_p^0\Pi v_i^{\alpha_i}\Pi u_j^{\beta_j}$ in $P(G_1)$ nor there was $\eta'_2 v_1^0v_2^0v_p^0\Pi v_i^{\alpha_i}\Pi u_j^{\beta_j}$ in $P(G_2)$.
\end{proof}
To close this section, we will prove the following theorem about the 3-path extendability of the class of graphs closely resembling ordinary wheels, but with the inner vertex "split in two".
\begin{theorem}\label{thm:lem3}
	Let $G$ be such that
	$$V(C(G)) = \lbrace v_1,\dots,v_k \rbrace,~~V(int(G))=\lbrace u,v \rbrace,$$
	and there is $3\le i\le k-1$ such that
	$$N_G(u)=\lbrace v,v_1,\dots,v_i \rbrace$$
	and
	$$N_G(v)=\lbrace u,v_i,\dots,v_k,v_1 \rbrace.$$
	There is a nonvanishing monomial $\eta v_1^0v_2^0v_k^0v^{\alpha_v}u^{\alpha_u} \Pi v_i^{\alpha_i}$ in $P(G-\overrightarrow{v_kv_1v_2})$ with $\alpha_i \le 2$ and $\alpha_v,\alpha_u\le 4$.
\end{theorem}
\begin{proof}
	To begin with, notice that $G'=G-\lbrace v_1, v_2, v_k \rbrace$ is an outerplanar near triangulation with $deg(v_3)=deg(v_{k-1})=2$, hence there is a non-vanishing monomial $\eta v_3^{\alpha_3}v_{k-1}^{\alpha_{k-1}}v^{\alpha_v}u^{\alpha_u} \Pi v_i^{\alpha_i}$ in $P(G')$ with $\alpha_3, \alpha_{k-1} \le 1$ and $\alpha_i, \alpha_v,\alpha_u \le 2$. As
	$$N_{G-\overrightarrow{v_kv_1v_2}}(v_1)=\lbrace v,u \rbrace,$$
	$$N_{G-\overrightarrow{v_kv_1v_2}}(v_2)=\lbrace v_3,u \rbrace$$
	and
	$$N_{G-\overrightarrow{v_kv_1v_2}}(v_k)=\lbrace v_{k-1},v \rbrace,$$
	in $P(G-\overrightarrow{v_kv_1v_2})$ there is a non-vanishing monomial $\eta v_1^0v_2^0v_k^0v_3^{\alpha_3+1}v_{k-1}^{\alpha_{k-1}+1}v^{\alpha_v+2}u^{\alpha_u+2} \Pi v_i^{\alpha_i}$, which is enough to prove the theorem.
\end{proof}

\begin{figure}[htb]
	\begin{center}
		\includegraphics[scale = 0.55]{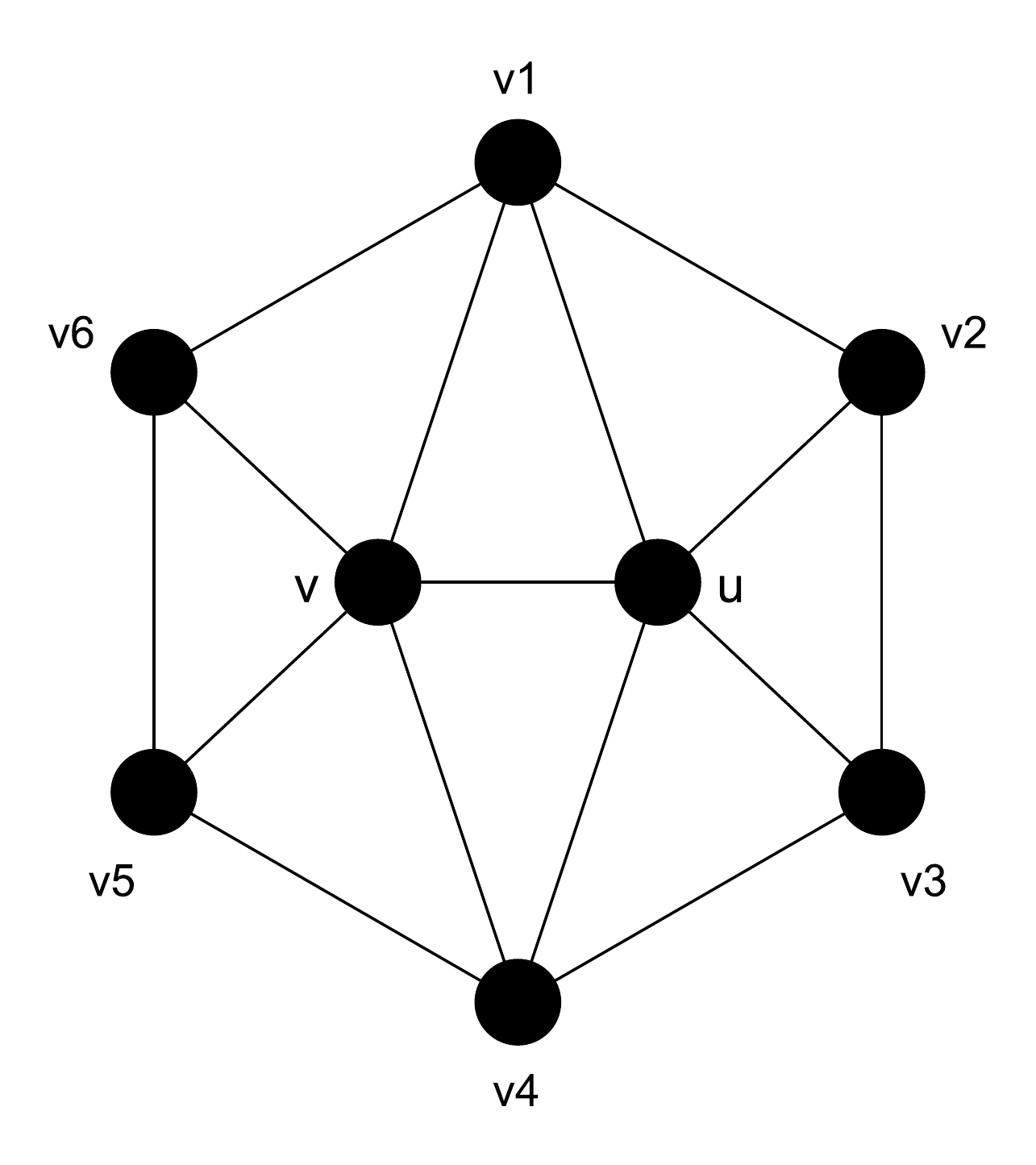}
	\end{center}
	\caption{An example of a graph to which Theorem~\ref{thm:lem3} applies.} \label{fig:lem3}
\end{figure}

\section{Main theorem}

We have now in our hands all the tools needed to prove the following polynomial analogue of Theorem~\ref{thm:thomassenExt}, which is the main result of this paper.
We have now in our hands all the tools needed to prove the following strengthening of Theorem~\ref{thm:thomassenExt} in language of graph polynomials, which is the main result of this chapter.
\begin{theorem} \label{thm:main}
	Let $G$ be a plane near-triangulation with outer cycle $C=(v_1,v_2, \dots, v_k,v_1)$ and interior vertices $u_1, \dots, u_l$. Then in $P(G-\overrightarrow{v_kv_1v_2})$ there is a non-vanishing monomial $\eta v_1^0v_2^0v_k^0\Pi v_i^{\alpha_i}\Pi u_j^{\beta_j}$ with $\alpha_i \le 2$ and $\beta_j \le 4$, unless $G$ contains a subgraph $G'$ which is a generalized wheel whose principal path
	is $\overrightarrow{v_kv_1v_2}$, and all vertices on the outer cycle of $G'$ are on $C$. 
\end{theorem}
\begin{proof}
	To prove this theorem suppose that the assertion is false, and let $G$ be the minimal counterexample. As for $k \le 5$ the theorem follows from Theorem~\ref{thm:smallcycle}, we may assume that $k > 5$. We also assume that $v_1, v_2$ and $v_k$ do not form a triangle in $G$ (i.e. $v_2 \nsim v_k$), as this would automatically rule out the existence of the monomial specified in the assertion (the triangle may also be viewed as a degenerate case of a broken wheel). Like in~\cite{ThomassenExt}, we will prove the theorem via a series of claims about the structure of $G$, largely identical to the ones in~\cite{ThomassenExt}.
	\begin{claim}\label{claim:c1}
		$C(G)$ has no chord.
	\end{claim}
	\begin{proof}
		Suppose $C(G)$ has a chord, and that the chord is not incident to $v_1$. Then we may cut $G$ along that chord, forming graphs $G', G''$, with $G'$ being the one containing $v_1, v_2$ and $v_k$ and $G''$ does not contain the edge that formed the chord in $G$. As $G$ is a minimal counterexample, there is a monomial in $P(G'-\overrightarrow{v_kv_1v_2})$ that adheres to the assertion, and by Theorem~\ref{thm:zhu} we can find suitable monomial in $P(G'')$ so that the exponents of the endpoints of the chords are both 0. Therefore in $P(G-\overrightarrow{v_kv_1v_2})$ there is a monomial $v_1^0v_2^0v_k^0\Pi v_i^{\alpha_i}\Pi u_j^{\beta_j}$ with $\alpha_i \le 2$ and $\beta_j \le 4$.
		
		Suppose now that there is a chord in $C(G)$ that is incident to $v_1$. Let $v_i$ be the other endpoint of that chord. Again split $G$ into $G'$ and $G''$ along the chord, in the way that $v_k \in V(G')$ and $v_2 \in V(G'')$. Notice, that as there is no generalized wheel with principal path $\overrightarrow{v_kv_1v_2}$, either there is no generalized wheel with principal path $\overrightarrow{v_kv_1v_i}$ in $G'$ or there is no generalized wheel with principal path $\overrightarrow{v_iv_1v_2}$ in $G''$. Without loss of generality we can assume the former. Then by minimality of $G$ there is a valid monomial with $v_1^0v_i^0v_k^0$ in $P(G'-\overrightarrow{v_kv_1v_i})$ and by Theorem~\ref{thm:zhu} there is a suitable monomial with $v_1^0v_2^0v_i^2$ in $P(G''-v_1v_2)$. Hence there is a monomial $v_1^0v_2^0v_k^0\Pi v_i^{\alpha_i}\Pi u_j^{\beta_j}$ with $\alpha_i \le 2$ and $\beta_j \le 4$ in $P(G-\overrightarrow{v_kv_1v_2})$, which finishes the proof.
	\end{proof}
	\begin{claim}\label{claim:c2}
		$G$ has no separating triangle and no separating $4$-cycle.
	\end{claim}
	\begin{proof}
		Assume that $G$ has a separating triangle $T$. By minimality of $G$, if we remove all vertices from the interior of $T$, the polynomial of the resulting graph will contain the requested monomial. By Theorem~\ref{thm:smallcycle_0}, the monomial of the removed subgraph contains a monomial such that the exponents of the vertices of the triangle are equal to 0, while all other are no larger than 4. Therefore, adding the interior of $T$ back, we can construct appropriate monomial in $P(G)$. Suppose now that there is a separating $4$-cycle $S$ in $G$. Remove from $G$ all the vertices and edges from the interior of $S$, and replace it with single edge $e$ (so that the resulting graph $G'$ is still a near-triangulation). Notice that if in  $P(G')$ there is a necessary monomial, then if we delete $e$ from $G'$ and add back the interior of $S$, then by Theorem~\ref{thm:smallcycle_0} there is a suitable monomial in $P(G)$. Suppose then, that there is no such monomial in $P(G')$. Then by the induction hypothesis $G'$ contains a generalized wheel. As we assume that $G$ does not contain a generalized wheel, this wheel needs to contain $e$. By Corollary~\ref{cor:gen-e} and Theorem~\ref{thm:zhu}, $P(G'-e)$ has a suitable monomial. Again it suffices to add back the interior of $S$ and apply Theorem~\ref{thm:smallcycle_0} to find a necessary monomial in $P(G)$ to finish the proof.
	\end{proof}
Here we will slightly deviate from the path followed by Thomassen in~\cite{ThomassenExt}. The reason is that there is no possibility to obtain a graph polynomial analogue of~\cite[Lemma 1]{ThomassenExt} --- graph polynomials either guarantee the existence of list-colouring, or give no information about the possibilities of extending the colourings. We therefore have to prove an additional claim.
\begin{claim} \label{lem:existence}
	Suppose $u$ is a vertex in $G \setminus C(G)$ which is joined to both of $v_i$, $v_j$ where $2 \le i \le j-2 \le k-2$ and let $G' = Int(v_j, v_{j+1}, \dots v_k, v_1, v_2, \dots, v_i, u, v_j)$. Then in $P(G'-\overrightarrow{v_kv_1v_2})$ there is a non-vanishing monomial $\eta v_1^0v_2^0v_k^0\Pi v_i^{\alpha_i}u^{\beta_u} \Pi u_j^{\beta_j}$ with $\alpha_i \le 2,\beta_u \le 3$ and $\beta_j \le 4$.
\end{claim}
We call each monomial existing by the above Claim a \emph{$u$-special} monomial.
\begin{proof}
	If $G'$ does not contain a generalized wheel as given in the assertion of the theorem, then $P(G')$ contains a non-vanishing monomial
	$$\eta v_1^0v_2^0v_k^0\prod v_i^{\alpha_i}u^{\beta_u} \prod u_j^{\beta_j}$$ with
	$$\alpha_i,\beta_u \le 2$$
	and
	$$\beta_j \le 4.$$
	We therefore assume that $G'$ contains such a wheel as a certain subgraph. Observe also, that as $C(G)$ had no chords, then any chord of $C(G')$ needs to be incident to $u$.
	
	Suppose at first that $C(G')$ has no chord. Then by the previous two claims it is an odd wheel. Let $u_0$ be the inner vertex of that wheel. As $G'-v_1$ is an outerplanar near-triangulation satisfying the conditions of Lemma~\ref{lem:030}, its polynomial contains a non-vanishing monomial $\eta v_2^0v_k^0\Pi v_i^2 u^3 u_0^2$. Now, $G'-\overrightarrow{v_kv_1v_2}$ can be constructed by taking $G'-v_1$ and adjoining $v_1$ to $u_0$, hence $P(G'-\overrightarrow{v_kv_1v_2})$ contains a non-vanishing monomial $\pm \eta v_1^0v_2^0v_k^0\Pi v_i^2 u^3 u_0^3$.
	
	If $v_1 \sim u$, then we may split $G'-\overrightarrow{v_kv_1v_2}$ by the chord $v_1u$ into $G'_1$ and $G'_2$, with $v_1u \in E(G'_1)$. By Theorem~\ref{thm:zhu}, $P(G'_1)$ contains a non-vanishing monomial $\eta_1 v_1^0v_k^0u^2\Pi v_i^{\alpha_i} \Pi u_j^{\beta_j}$ with $\alpha_i \le 2$ and $\beta_j \le 4$, while $P(G'_2)$ contains a non-vanishing monomial $\eta_1 v_1^0v_2^0u^1\Pi v_i^{\alpha_i} \Pi u_j^{\beta_j}$ with $\alpha_i \le 2$ and $\beta_j \le 4$. Hence in $P(G'-\overrightarrow{v_kv_1v_2})$ there is a monomial $\eta_1 v_1^0v_2^0v_k^0u^3\Pi v_i^{\alpha_i} \Pi u_j^{\beta_j}$ with $\alpha_i \le 2$ and $\beta_j \le 4$, and it does not vanish by Corollary~\ref{cor:00_2}.
	
	The final case is when $v_1 \nsim u$, and $C(G')$ has at least one chord. Let $p$ be the lowest such that $v_p \sim u$ and $q$ be the largest that $v_q \sim u$ (possibly $p=i$ or $q=j$). Now consider:
	$$G''=Int(v_q, \dots v_k, v_1, v_2, \dots, v_p, u, v_q).$$
	This has no chords, so going back to the first case we already know that $P(G''-\overrightarrow{v_kv_1v_2})$ contains a non-vanishing monomial $\eta v_1^0v_2^0v_k^0\Pi v_i^2 u^3 u_0^3$. It only remains to reconstruct $G'$ by adding the parts that were cut out by the chords (one of them possibly empty) using Theorem~\ref{thm:zhu}.
\end{proof}
\begin{claim} \label{claim:c3}
	If $u$ is a vertex in $G \setminus C(G)$ which is joined to both of $v_i$, $v_j$ where $2 \le i \le j-2 \le k-2$ then $u$ is joined to each of $v_i, v_{i+1}, \dots v_j$.
\end{claim}
\begin{proof}
	Consider a split of $G$ into
	$$G' = Int(v_j, v_{j+1}, \dots v_k, v_1, v_2, \dots, v_i, u, v_j)$$
	and
	$$H' = Int(u, v_i, v_{i+1}, \dots v_j, u).$$
	As $H'$ is not a counterexample to Theorem~\ref{thm:main}, by Lemma~\ref{lem:existence} $H'$ contains a generalized wheel with principal path $\overrightarrow{v_j u v_i}$ with all vertices on the cycle $(u, v_i, v_{i+1}, \dots v_j,u)$.
	
	We show, that $H'$ is a broken wheel. Suppose otherwise, let $p, q$ be such that $i \le p \le q - 2 \le j-2$ and there is no edge between $u$ and $v_r$ for $p < r < q$. Then consider a split of $G$ into
	$$G'' = Int(v_q, v_{q+1}, \dots v_k, v_1, v_2, \dots, v_p, u, v_q)$$
	and
	$$H = Int(u, v_p, v_{p+1}, \dots v_q,u).$$
	Again, $H$ is not a counterexample to Theorem~\ref{thm:main}. Hence, by Claim~\ref{lem:existence} $H$ contains a generalized wheel with principal path $\overrightarrow{v_q u v_p}$ with all vertices on the cycle $(u, v_p, v_{p+1}, \dots v_q,u)$. But this means, by the minimality of $G$ and Theorem~\ref{thm:smallcycle}, that $H$ itself is an odd wheel.
	
	Because $G$ is a counterexample to Theorem \ref{thm:main}, then in the polynomial $$P(G-\overrightarrow{v_kv_1v_2}) = P(G''-\overrightarrow{v_kv_1v_2}) P(H - \overrightarrow{v_q u v_p})$$ each monomial of the form $\eta v_1^0v_2^0v_k^0\Pi v_i^{\alpha_i}\Pi u_j^{\beta_j}$ with $\alpha_i \le 2$, $\beta_j \le 4$ vanishes. But this means that the polynomial $P(H - \overrightarrow{v_q u v_p})$ cancels (or extends its exponents too much) every $u$-special monomial of $P(G''-\overrightarrow{v_kv_1v_2})$ existing by Claim~\ref{lem:existence}. 
	
	Using Lemma~\ref{lem:ordinary2k1} one may notice which monomials can be cancelled. At first, to be cancelled by monomials in $P(H - \overrightarrow{v_q u v_p})$ having 1 as a sum of exponents for $v_q, u$ and $v_p$ the polynomial of all $u$-special monomials of $P(G''-\overrightarrow{v_kv_1v_2})$ is of the form
	$$ S_u = \sum_{\alpha, \beta, \gamma} (u^\alpha v_p^\beta v_q^\gamma - u^{\alpha+1} v_p^{\beta-1} v_q^\gamma - u^{\alpha+1} v_p^\beta v_q^{\gamma-1}) Q(\alpha, \beta, \gamma),$$
	where the sum is over some $\alpha \le 3$, $1 \le \beta, \gamma \le 2$, while $Q(\alpha, \beta, \gamma)$ is a polynomial composed of monomials having the form
	$$v_k^0 v_1^0 v_2^0 \prod_{2 < i < p} v_i^{\alpha_i} \prod_{q <i < k} v_i^{\alpha_i} \prod_{u_j \neq u} u_j^{\beta_j},$$
	with $\alpha_i \le 2$, $\beta_j \le 4$. 
	
	Now notice, that monomials with $\alpha \le 2$ or $\beta + \gamma \le 3$ are not cancelable by monomials of $P(H - \overrightarrow{v_q u v_p})$ having 2 as a sum of exponents for $v_q, u$ and $v_p$, hence the sum in $S_u$ has only one summand --- for $\alpha = 3$, $\beta = \gamma = 2$. But, by Observation~\ref{lem:broken5}, a broken wheel on 5 vertices cancels (or extend its exponents too much) all the monomials from $S_u$ as well. Moreover, it has fewer vertices. Hence, by exchanging graph $H$ by a broken wheel on 5 vertices one may obtain a smaller counterexample, and this contradiction concludes the proof. 
\end{proof}
\begin{claim}\label{claim:c4}
	There is no interior vertex $u_j$ neighboring both $v_2$ and $v_k$.
\end{claim}
\begin{proof}
	Suppose there is a vertex $u$ neighboring both $v_2$ and $v_k$. As we established that $v_2$ cannot neighbor $v_k$ and $G$ has no separating $4$-cycle, $u$ has to neighbor $v_1$ as well. By the previous claim, $u$ also neighbors all vertices $v_i, 3 \le i < k$. Therefore $G$ forms an ordinary wheel. By the assumption, this wheel cannot be an odd wheel, hence it's an even wheel. But even wheels contain an appropriate monomial in their polynomials, so we reach a contradiction.
\end{proof}
\begin{claim}\label{claim:c56}
	Both $v_3$ and $v_{k-1}$ have degree at least $4$.
\end{claim}
\begin{proof}
	We will prove the claim only for $v_3$, as the case of $v_{k-1}$ is symmetric.
	
	Suppose $deg(v_3)<4$. As $C(G)$ has no chord, $deg(v_3)=3$. Let
	$$G'=G-\overrightarrow{v_kv_1v_2}$$
	and let $u$ be the third neighbor of $v_3$, apart from $v_2$ and $v_4$. We have two cases to cover:
	
	\textit{Case 1. G'$-v_3$ fulfills the conditions of the theorem: }As $G'-v_3$ fulfills the conditions of the theorem, there is a monomial $\eta v_1^0v_2^0v_k^0u^{\alpha}\Pi v_i^{\alpha_i}\Pi u_j^{\beta_j}$ with $\alpha,\alpha_i \le 2$ and $\beta_j \le 4$ in $P(G'-v_3)$. We look for monomial $\eta v_1^0v_2^0v_k^0\Pi v_i^{\alpha_i}\Pi u_j^{\beta_j}$ with $\alpha_i \le 2$ and $\beta_j \le 4$ in $P(G')$ Obviously, $$P(G')=P(G'-v_3)(v_3-v_2)(v_3-v_4)(v_3-u)=$$
	$$=P(G'-v_3)(v_3^3-v_2v_3^2-v_3^2v_4-v_3^2u+v_3v_4u+v_2v_3u+v_2v_3v_4-v_2v_4u)=$$ 
	$$=Q(G')-P(G'-v_3)v_3^2v_4-P(G'-v_3)v_3^2u+P(G'-v_3)v_3v_4u.$$
	The 5 monomials that $Q(G')$ comprises of are not suitable to produce a necessary monomial as they have either $v_3$ in power $3$, or $v_2$ in power higher than $0$. Of course, we can find a suitable monomial in $P(G'-v_3)v_3^2u$, therefore it occurs in $P(G')$. If it does not vanish, then the case is done. If it does vanish, then in $P(G'-v_3)$ there is a monomial $-\eta v_1^0v_2^0v_k^0u^{\alpha+1}v_4^{\alpha_4-1}\Pi v_i^{\alpha_i}\Pi u_j^{\beta_j}$, and the monomial we found in $P(G'-v_3)v_3^2u$ is also in $P(G'-v_3)v_3^2v_4$, with opposite sign. But as in $P(G-v_3)$ there is a monomial
	$$-\eta v_1^0v_2^0v_k^0u^{\alpha+1}v_4^{\alpha_4-1}\prod v_i^{\alpha_i}\prod u_j^{\beta_j},$$ we can also find a monomial that satisfies the condition in $P(G'-v_3)v_3v_4u$, and as it is the only section of $P(G')$ that contains monomials with $v_3$ in power $1$, the monomial does not vanish in $P(G')$, and the case is finished.
	
	\textit{Case 2. G'$-v_3$ does not fulfill the conditions of the theorem: } If so, then there is a generalized wheel (with principal path $\overrightarrow{v_kv_1v_2}$) in $G'-v_3$. Let $q$ be the largest such that $v_q \sim u$. By Claim~\ref{claim:c3}, $u$ also neighbors all $v_i$, $4 \le i < q$. As $C(G)$ has no chords, there is a vertex $u_0$ in the interior of $C(G')$ that neighbors $u, v_1$ and all $v_i, q \le i \le k$. By the preceding claim we have that $v_2 \nsim u_0$, therefore $v_1 \sim u$. As a result, we see that $G$ adheres to the assumptions of Theorem~\ref{thm:lem3}, hence $P(G)$ contains a necessary monomial.
\end{proof}
\begin{claim}\label{claim:c7}
	There is a vertex $u$ in $int(C)$ that neighbors both $v_3$ and $v_{k-1}$.
\end{claim}
\begin{proof}
	Suppose otherwise. Take
	$$G'=G-\overrightarrow{v_kv_1v_2}-v_3,$$
	let $u_1,\dots,u_m$ be the interior neighbors of $v_3$ from $G$ in counterclockwise order (by previous claim $m \ge 2$) and $u_i, i>m$ be the interior vertices of $G'$. Observe that as $C(G)$ had no chords, if $G'$ contained a generalized wheel as in the assertion of the theorem, it would be either an odd wheel, or a chain of wheels, with the part containing $v_{k-1}$ also being an odd wheel. Therefore, $v_{k-1}$ would necessary have degree equal to $3$, and that would contradict the preceding claim. Hence $G'$ has no such generalized wheel, and as a result, in $P(G')$ there is a non-vanishing monomial
	$$\eta v_1^0v_2^0v_k^0 \prod v_i^{\alpha_i} \prod_{j=1}^{m} u_j^{\beta_j} \prod_{j>m} u_j^{\gamma_j},$$
	with $\alpha_i,\beta_j \le 2$ and $\gamma_j \le 4$. Now we attach $v_3$ back to $v_2,u_1,\dots,u_m,v_4$. The graph we added to $G'$ is $S_{m+2}$ and we know its polynomial contains a non vanishing monomial $v_2^0v_3^2v_4^0 \Pi u_i^1$ (we choose a suitable ordering of vertices in $S_{m+2}$). This monomial, combined with the monomial specified earlier, gives us a monomial
	$$\eta v_1^0v_2^0v_k^0 \prod v_i^{\alpha_i} \prod_{j=1}^{m} u_j^{\beta_j} \prod_{j>m} u_j^{\gamma_j},$$
	with $\alpha_i \le 2$, $\beta_j \le 3$ and $\gamma_j \le 4$, in $P(G-\overrightarrow{v_kv_1v_2})$. If this monomial does not vanish in $P(G-\overrightarrow{v_kv_1v_2})$, then we reach the contradiction and the proof is complete. If it does vanish, in $P(G')$ there is a non-vanishing monomial
	$$\eta v_1^0v_2^0v_k^0 \prod v_i^{\alpha_i} \prod_{j=1}^{m} u_j^{\beta_j} \prod_{j>m} u_j^{\gamma_j},$$
	with
	$$\alpha_4 \le 1, \alpha_i \le 2, \beta_j \le 3, \gamma_j \le 4.$$
	There is a monomial $-v_2^0v_3^1v_4^1 \Pi u_i^1$ in $P(S_{m+2})$, and this is the only monomial there with $v_2^0v_3^1$. Hence there is a monomial $-\eta v_1^0v_2^0v_3^1v_k^0 \Pi v_i^{\alpha_i} \Pi u_j^{\beta_j}$, with $\alpha_i \le 2$ and $\beta_j \le 4$, in $P(G-\overrightarrow{v_kv_1v_2})$, and this contradiction concludes the proof.
\end{proof}
We are now set up for a contradiction of the initial assumption. Let $u_1,\dots,u_m$ be the neighbors of $v_3$ in the counterclockwise order. Suppose at first, that $G-v_3$ has no generalized wheel as an appropriate subgraph. We can then find a necessary monomial by the argument identical to that from the preceding claim. We then assume that $G-v_3$ contains a generalized wheel.

As we established, there is a vertex in the interior of
$$G'=G-\overrightarrow{v_kv_1v_2}$$
that neighbors both $v_3$ and $v_{k-1}$. By Claim~\ref{claim:c3}, we also know that this vertex needs to neighbor $v_4,\dots,v_{k-2}$. As $G$ has no separating triangles, this vertex has to be $u_m$, since if there was $q<m$ such that $v_{k-1} \sim u_q$, then $v_3v_4u_q$ would form a separating triangle. Suppose now that $v_k$ neighbors any of the $u_1,\dots,u_{m-1}$. Then by Claim~\ref{claim:c3} that vertex would also neighbor $v_4,\dots,v_{k-1}$, which is clearly impossible as $v_{k-1}\sim u_m$. Now assume that $v_k \sim u_m$. By Claim~\ref{claim:c56} $v_{k-1}$ has a neighbor other than $u_m$ in the interior of $C(G)$, hence either $v_{k-1}v_ku_m$ or $v_{k-2}v_{k-1}u_m$ would form a separating triangle. Therefore $v_k \nsim u_m$.

Observe that as $v_{k-1}u_m$ forms a chord in $C(G-v_3)$, if
$$G''=Int(v_1,v_2,u_1,\dots,u_m,v_{k-1},v_k,v_1)$$
had a monomial
$$\eta v_1^0v_2^0v_k^0v_{k-1}^{\alpha_{k-1}}\prod u_j^{\beta_j}, \alpha_{k-1}, \beta_j \le 2,$$
in $P(G''-\overrightarrow{v_kv_1v_2})$, then by Theorem~\ref{thm:zhu} we could find a monomial
$$\eta' v_1^0v_2^0v_k^0\prod v_i^{\alpha_i}\prod_{j \le m} u_j^{\beta_j} \prod_{j > m} u_j^{\gamma_j}$$
with $\alpha_i,\beta_j \le 2, \gamma_j \le 4$, in $P(G'-v_3)$, which in turn would enable us to use the reasoning from the proof of Claim~\ref{claim:c7}. Therefore $P(G''-\overrightarrow{v_kv_1v_2})$ has no such monomial, and by the induction hypothesis $G''$ contains a generalized wheel. As $C(G)$ has no chords, there is a vertex $u_0$ in $int(G'')$ neighboring $u_m,v_{k-1},v_k$ and $v_1$, and as $v_1 \sim u_0$, by Claim~\ref{claim:c4} we have that $v_2 \nsim u_0$. Notice, that if we repeat the reasoning above with $v_{k-1}$ playing the role of $v_3$, we would come to the conclusion, that there is a vertex $u'_0$ in $int(v_1,v_2,v_3,u'_{m'},\dots,u'_1,v_k,v_1)$ neighboring $u'_{m'},v_1,v_2$ and $v_3$. Hence $m=2$ and $u'_0=u_1$, which implies $v_1 \sim u_1$. Moreover, as $G$ is a near-triangulation and has no separating 4-cycle, either $u_0 \sim u_1$ or $v_1 \sim u_2$. But the latter is impossible, as then $v_1,v_2,v_3,u_2$ and $v_1,u_2,v_{k-1},v_k$ would form separating 4-cycles. Finally, we remind that by Claim~\ref{claim:c3}, $u_2$ neighbors all of $v_4, \dots, v_{k-2}$. The structure of $G$ is visualized on Figure~\ref{fig:final}.

We now have to show that $P(G-\overrightarrow{v_kv_1v_2})$ contains a non-vanishing monomial $\eta v_1^0v_2^0v_k^0\Pi v_i^{\alpha_i} \Pi u_j^{\beta_j}$, with $\alpha_i \le 2$ and $\beta_j \le 4$. Considering the structure of $G$, namely that
$$N_{G-\overrightarrow{v_kv_1v_2}}(v_1)=\lbrace u_0,u_1 \rbrace,$$
$$N_{G-\overrightarrow{v_kv_1v_2}}(v_2)=\lbrace v_3,u_1 \rbrace$$
and
$$N_{G-\overrightarrow{v_kv_1v_2}}(v_k)=\lbrace v_{k-1},u_0 \rbrace,$$
it suffices to show that the polynomial of
$$G_0=G-\lbrace v_1,v_2,v_k \rbrace$$
contains a non-vanishing monomial $\eta \Pi v_i^{\alpha_i} \Pi u_j^{\beta_j}$, where
$$\alpha_3, \alpha_{k-1} \le 1, \alpha_i,\beta_0,\beta_1 \le 2, \beta_2 \le 4.$$ By the initial assumption we have $k>5$ (also, for $k=5$ there would be a separating 4-cycle $u_0u_1v_3v_4$ in $G$). Let
$$G_1=G_0-u_0.$$
As $k>5$, $G_1$ satisfies the assumptions of Lemma~\ref{lem:final}, therefore $P(G_1)$ contains a non-vanishing monomial
$$\eta u_1^1u_2^3v_3^1v_{k-1}^0\prod v_i^2.$$
As a result,
$$P(G_0)=P(G_1)(u_1-u_0)(u_2-u_0)(v_{k-1}-u_0)$$
contains a monomial
$$\eta u_0^0u_1^2u_2^4v_3^1v_{k-1}^1\prod v_i^2$$
which is what we needed to obtain the contradiction.
\end{proof}

\begin{figure}[htb]
	\begin{center}
		\includegraphics[scale = 0.55]{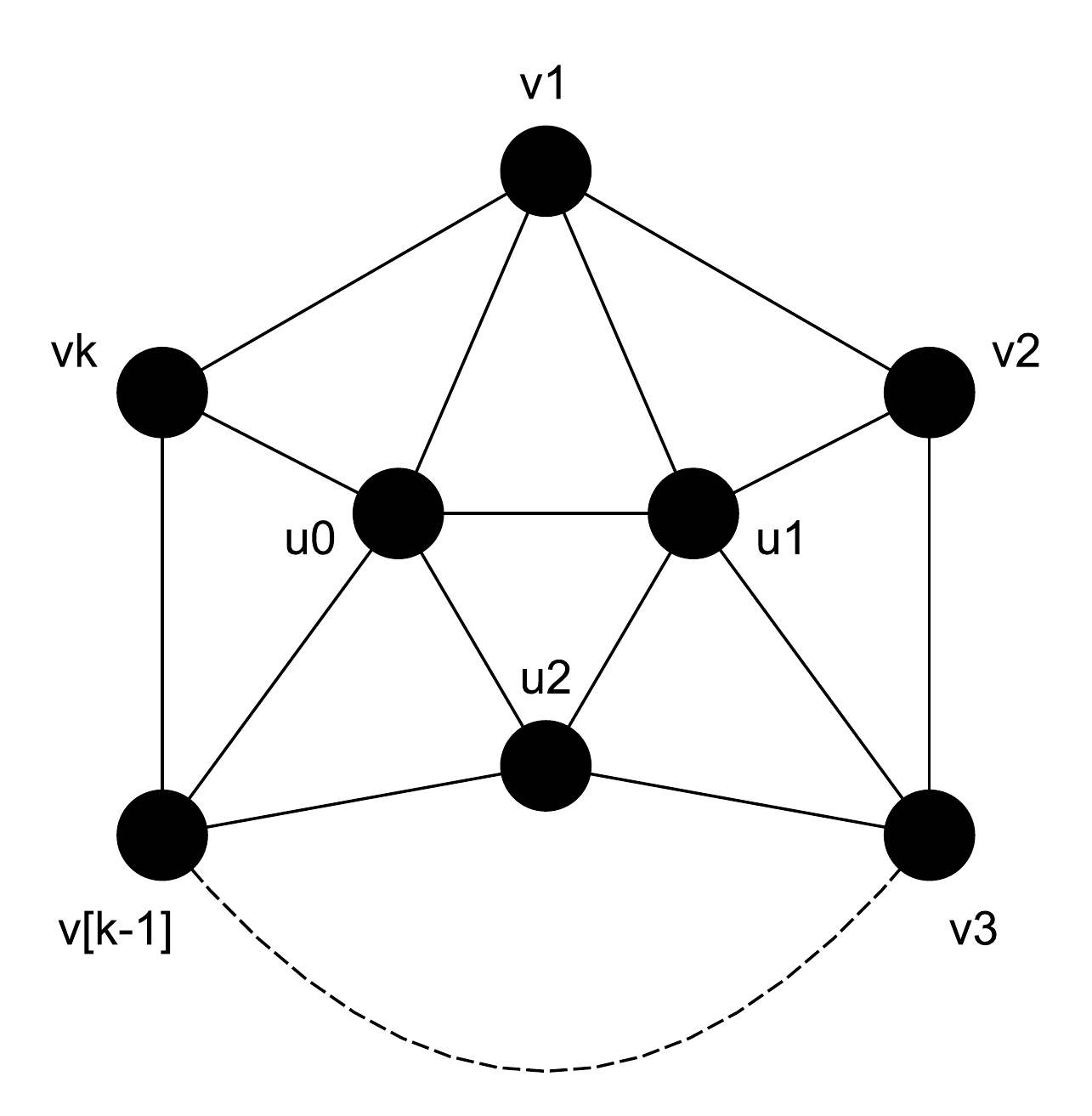}
	\end{center}
	\caption{A structure of the hypothetical minimum counterexample to Theorem~\ref{thm:main} forced by Claims~\ref{claim:c1}~---~\ref{claim:c7}. Vertex $u_2$ neighbours all the vertices between $v_3$ and $v_{k-1}$.} \label{fig:final}
\end{figure}

\end{document}